\documentclass{amsart}

\usepackage{amsmath,amssymb}
\usepackage{graphicx}
\usepackage{epstopdf}

\newcommand{\edge}[1]{\stackrel{#1}{\rule[3pt]{2em}{.5pt}}}
\newcommand{\alb}{\mathsf{X}}
\newcommand{\zero}{\mathit{0}}
\newcommand{\one}{\mathit{1}}
\newcommand{\two}{\mathit{2}}
\newcommand{\unit}{\varepsilon}
\newcommand{\xs}{\mathsf{X}^*}
\newcommand{\xo}{\mathsf{X}^\omega}
\newcommand{\arr}{\longrightarrow}
\newcommand{\R}{\mathbb{R}}
\newcommand{\Z}{\mathbb{Z}}
\newcommand{\Q}{\mathbb{Q}}

\newcommand{\X}{\mathcal{X}}
\newcommand{\symm}{\mathsf{S}}
\newcommand{\alt}{\mathsf{A}}
\newcommand{\full}{\mathsf{F}}

\newcommand{\mapdown}[1]%
{\Big\downarrow\rlap{$\vcenter{\hbox{$\scriptstyle#1$}}$}}

\newtheorem{theorem}{Theorem}[section]
\newtheorem{lemma}[theorem]{Lemma}
\newtheorem{proposition}[theorem]{Proposition}
\newtheorem{corollary}[theorem]{Corollary}

\theoremstyle{definition}
\newtheorem{defi}[theorem]{Definition}
\newtheorem{example}[theorem]{Example}
\newtheorem*{problem}{Problem}

\title[Palindromic subshifts and simple groups of intermediate growth]{Palindromic subshifts and simple periodic groups of intermediate growth}
\author{Volodymyr Nekrashevych}

\begin{document}

\maketitle

\begin{abstract}
We describe a new class of groups of Burnside type, giving a procedure transforming an arbitrary non-free minimal action of the dihedral group on a Cantor set into an orbit-equivalent action of an infinite finitely generated periodic group. We show that if the associated Schreier graphs are linearly repetitive, then the group is of intermediate growth.
In particular, this gives first examples of simple groups of intermediate growth.
\end{abstract}

\section{Introduction}

Let $a$ be a homeomorphism of period two (an involution) of a Cantor set $\X$. Choose a finite group $A$ of homeomorphisms of $\X$ such that for all $h\in A$, $\zeta\in\X$ we have $h(\zeta)=\zeta$ or $h(\zeta)=a(\zeta)$, and for every $\zeta\in\X$ there exists $h\in A$ such that $h(\zeta)=a(\zeta)$. We say that $A$ is a \emph{fragmentation} of $a$.

Suppose that we have a minimal action on the Cantor set of the infinite dihedral group $D_\infty$ generated by two involutions $a, b$. (An action is said to be \emph{minimal} if all its orbits are dense.) Let $A$ and $B$ be fragmentations of $a$ and $b$. We are interested in the group $\langle A\cup B\rangle$ generated by $A$ and $B$.

Examples of such groups are the first Grigorchuk
group~\cite{grigorchuk:80_en} and every group from the family of Grigorchuk groups defined in~\cite{grigorchuk:growth_en}. They are all obtained by fragmenting one particular minimal action of the dihedral group (associated with the binary Gray code, see Example~\ref{ex:Grigorchuk} below).

We show that under rather general conditions on the fragmentation, the group $G=\langle A\cup B\rangle$ possesses very interesting properties, pertinent to three classical problems of group theory (Burnside's problem on periodic groups, Day's problem on amenable groups, and Milnor's problem on intermediate growth). For example, \emph{every} non-free minimal action of $D_\infty$ can be fragmented to produce a finitely generated infinite periodic group. Moreover, if the action of $D_\infty$ is expansive, then one can fragment it (in uncountably many ways) to get a simple group. Actions of low complexity (for example coming from palindromic minimal substitutional subshifts) can be fragmented to produce groups of intermediate growth, including simple ones.

Namely, we prove the following (see Theorem~\ref{th:main}).

\begin{theorem}
\label{th:theoremone}
Suppose that $\xi$ is a fixed point of $a$, and that for every $h\in A$ such that $h(\xi)=\xi$, the interior of the set of fixed points of $h$ accumulates on $\xi$. Then the group $\langle A\cup B\rangle$ is periodic.
\end{theorem}

If the action of $D_\infty=\langle a, b\rangle$ is expansive, then the \emph{topological full group} of $G=\langle A\cup B\rangle$ contains an infinite finitely generated simple subgroup $\alt(G, \X)$, see~\cite{nek:fullgr}. The group $\alt(G, \X)$ is a subgroup of a (possibly bigger that $G$) fragmentation $\langle A_1\cup B_1\rangle$ to which Theorem~\ref{th:theoremone} is also applicable, so that the group $\alt(G, \X)$ is periodic. For a definition of the topological full group and the group $\alt(G, \X)$, see Subsection~\ref{ss:full}.

Here a group $G$ is said to be \emph{periodic} if for every $g\in G$ there exists $n$ such that $g^n$ is the identity. The question of existence of finitely generated infinite periodic groups (groups of \emph{Burnside type}) is the classical \emph{Burnside problem}~\cite{burnside:problem}. A harder version (\emph{bounded Burnside problem}) asks for a group with bounded order $n$ of elements.
The general problem (without a bound on the order of elements) was solved by 1964 by E.~Golod and I.~Shafarevich~\cite{golod:nilalgebras}.
The bounded Burnside problem was solved by S.~Adyan and P.~Novikov~\cite{novikovadjan} in 1968. The \emph{restricted Burnside problem} (the bounded Burnside problem in the class of residually finite groups) was solved (in the negative) in 1989 by E.~Zelmanov~\cite{zelmanov:burnside}. For a survey of the Burnside problem and related topics, see~\cite{grigorchlysenok:burnside}.

Previously known examples of infinite periodic finitely generated groups can be split into three classes. One class consists of solutions of the bounded Burnside problem.
They are constructed using some versions of the small cancellation theory: combinatorial as in the original Adyan-Novikov proof (see~\cite{adian:book}), geometrical due to A.~Olshanskii (see~\cite{olshanski:book,ivanov:burnside,lysenok:burnside}) and E.~Rips~\cite{rips:cancellation}, and via M.~Gromov's theory of hyperbolic groups (see~\cite{gro:hyperb,olshanski:residualing}).

Second class are the Golod-Shafarevich groups.
Third class are groups generated by automata and groups defined by their action on rooted trees. The first examples in this class were constructed by S.~Aleshin~\cite{al:burn_en}, V.~Sushchanskii~\cite{sushch:periodic}, R.~Grigorchuk~\cite{grigorchuk:80_en}, and N.~Gupta and S.~Sidki~\cite{gupta-sidkigroup}. Related constructions (self-similar groups, branch groups, etc.) became a very active area of research, see~\cite{handbook:branch,gri:solvedunsolved,nek:book}. But periodic groups in this class remain to be more or less isolated examples.

Note that the groups in the latter two classes are necessarily residually finite. In fact, most methods of study of these groups heavily rely on their residual finiteness (e.g., on their action on a rooted tree). The groups in the first class may be simple, e.g., the Olshanskii-Tarski monsters~\cite{olshanski:monster}.

Theorem~\ref{th:theoremone} produces a new large class of groups of Burnside type. It includes the Grigorchuk group, so it intersects with the above mentioned third class of periodic groups, but it also contains many simple groups. For instance, we produce first example of groups of Burnside type generated by piecewise isometries of a polygon (with a finite number of pieces).

Any non-free (i.e., such that some non-trivial elements have fixed points) action of $D_\infty$ can be fragmented so that it satisfies the conditions of Theorem~\ref{th:theoremone}. For example, if $\mathcal{S}\subset\alb^{\Z}$ is a minimal \emph{palindromic} subshift (such that elements of $\mathcal{S}$ contain arbitrary long palindromes), then the transformations
\[a(w)(n)=w(-n),\qquad b(w)(n)=w(1-n),\]
for $w\in\mathcal{S}$ generate a minimal action of $D_\infty$ such that $a$ or $b$ (depending on parity of the lengths of arbitrarily long palindromes) have a fixed point. Then it is easy to fragment the corresponding generator so that the conditions of Theorem~\ref{th:theoremone} are satisfied. After that, after passing to the group $\alt(G, \mathcal{S})$, we get a finitely generated simple periodic group.

Minimal palindromic shifts are classical objects in Dynamics, see, for example~\cite[Sections~4.3--4]{baakegrimm} and references therein. See also~\cite{GLN}, where spectral properties of a substitutional system associated with the Grigorchuk group are studied.

Answers to the Burnside problem were important examples for the theory of
amenable groups. Amenability was defined by J.~von~Neumann~\cite{vneumann:masses} in his analysis of the Banach-Tarski paradox. He noted that a group containing a non-commutative free group is non-amenable, and showed that amenability is preserved under some group-theoretic operations (extensions, direct limits, passing to a subgroup and to a quotient). So, there are ``obviously non-amenable'' groups (groups containing a free subgroup), and ``obviously amenable'' groups (groups that can be constructed from finite and commutative groups using the above operations). The ``obviously amenable'' groups are called \emph{elementary amenable}. More on amenability,
see~\cite{wagon:banachtarski,paterson:amenability,greenleaf}.

Groups of Burnside type are never obviously non-amenable, since they do
not contain free subgroups. They are also never elementary amenable,
see~\cite[Theorem~2.3]{chou:elementaryamenable}. The fact
that the class of groups without free subgroups and the class of
elementary amenable groups are distinct is proved
in~\cite{chou:elementaryamenable} precisely using the existence of
groups of Burnside type.

Groups of Burnside type were the first examples to show that neither class
(groups without free subgroups and elementary amenable groups)
coincides with the class of amenable groups.

They were the first examples
of non-amenable groups without free subgroups (free Burnside groups
and Tarski monsters, see~\cite{olshanskii:nonamenable,adyan:randomwalk}).
All known examples of finitely generated infinite groups of bounded exponents are non-amenable. A positive answer to an open question by Y.~Shalom~\cite{shalom:congress} (whether free Burnside groups have property T) would imply that all of them are non-amenable.

The groups in the second class (the Golod-Shafarevich groups) are all non-amenable by a result of M.~Ershov~\cite{ershov:gs}.

Groups of Burnside type (the Grigorchuk groups~\cite{grigorchuk:milnor_en,grigorchuk:growth_en}) were also the first examples of non-elementary amenable groups.
In fact, for a long time the only known examples of non-elementary amenable groups were based on the Grigorchuk groups. Later, other examples were constructed~\cite{barthvirag,bkn:amenability,amirangelvirag:linear}, but all of them where defined by their actions on rooted trees, so, in particular, they were residually finite.
For a long time the question if there exist infinite finitely generated simple amenable groups was open.
It was answered by K.~Juschenko and N.~Monod in~\cite{juschenkomonod}. They showed that the \emph{topological full
group} of a minimal homeomorphism of the Cantor set is amenable (confirming a conjecture of R.~Grigorchuk and K.~Medynets). Here
the topological full group of a (cyclic in this case) group $G$ acting on a Cantor set $\X$ is
the group of all homeomorphisms $h:\X\arr\X$ such that for every
$\zeta\in\X$ there exists a neighborhood $U$ of $\zeta$ and an element
$g\in G$ such that $g|_U=h|_U$. In other words, it is obtained by ``fragmenting'' a minimal action of $\Z$ in a way similar to our definition of a fragmentation of $D_\infty$. Our definition is different, however, as we do not require the sets where the action of an element $h\in A$ coincides with the action of $a$ to be open.

It was proved earlier in~\cite{matui:fullI} and~\cite{bezuglyiMedynets:fullgroup} that if $\tau$ is a minimal
homeomorphism of the Cantor set, then the topological full group
of $\langle\tau\rangle$ has simple derived subgroup,
and if the
homeomorphism is expansive (i.e., is conjugate to a subshift), then the derived subgroup is finitely generated. See also~\cite{ChJN}, where a similar result is proved for $\Z^n$-actions. H.~Matui proved in~\cite{matui:expgrowth} that the derived subgroups of the full groups of minimal subshifts are of exponential growth.

The methods of~\cite{juschenkomonod} were generalized in~\cite{YNS} to cover a wide class of non-elementary amenable groups, including almost all known examples.

For arbitrary fragmentations $A, B$ of the generators $a, b$ of a minimal action of the dihedral group, the group $\langle A\cup B\rangle$ can be embedded into the topological full group of a minimal subshift. This was observed for the first time (for the Grigorchuk group) by N.~Matte~Bon in~\cite{mattebon:grigochuk}. It follows that all groups generated by fragmenting a minimal dihedral group are amenable. Theorem~\ref{th:theoremone} produces, therefore, the first examples of simple amenable groups of Burnside type. The only previously known simple groups of Burnside type were of bounded exponents and non-amenable (e.g., Olshanskii-Tarski monsters~\cite{olshanski:monster}).

If $G$ is a group generated by a finite set $S$, then its \emph{growth function} $\gamma(n)$ is the number of elements of $G$ that can be written as products of at most $n$ elements of $S\cup S^{-1}$. J.~Milnor asked in~\cite{milnor:5603} whether there exists a group with growth function eventually bigger than any polynomial and eventually smaller than any exponential function. Such groups are called groups of \emph{intermediate growth}, and the first example of such a group is the Grigorchuk group from~\cite{grigorchuk:80_en,grigorchuk:milnor_en}. Amenability of the Grigorchuk group follows from its intermediate growth.
Until now all known groups of intermediate growth were either based on the Grigorchuk's examples, or were very close to it (see~\cite{barthsunik,buxperez:imgi}). See, for example~\cite{erschler:piecwise,kassabovpak:oscillating,brieussel} where the Grigorchuk group is used to construct groups with different growth types and boundary behavior.

Recently, L.~Bartholdi and A.~Erschler developed a technique of \emph{inverted orbits}, and used it to construct a great variety of groups with a prescribed intermediate growth, see~\cite{BE:permextensions,BE:givengrowth,bartholdiErschler:imbeddings} and~\cite{barth:growthsurvey}. 

All these constructions of groups of intermediate growth used as a starting point the groups from the family of Grigorchuk groups defined in~\cite{grigorchuk:growth_en}. This imposes some restrictions on the type of groups that can be obtained this way. In particular, the following problem, asked by R.~Grigorchuk in 1984 (see Problem 9.8 in the ``Kourovka notebook''~\cite{kourovka}) remained to be open.

\begin{problem}
\label{prob:grigorchuk}
Does there exist a finitely generated simple group of intermediate growth?
\end{problem}

See also~\cite[page~132]{mann:growth}, \cite[Problem~15.17]{kourovka}, \cite[Problem~2]{gri:gap}, and~\cite{bajorskamacedonska,bartholdiErschler:imbeddings}.
In fact, it was even an open question for a long time whether all groups of intermediate growth are residually finite (see~\cite[8.4]{grigorchuk:kyoto}). Note that it follows from M.~Gromov's Theorem~\cite{gro:gr} that groups of polynomial growth are residually finite. The first examples of groups of intermediate growth that are not residually finite were constructed by A.~Erschler in~\cite{erschler:nonresfin}. L.~Bartholid and A.~Erschler constructed a group of intermediate growth containing a simple (locally finite) subgroup~\cite{bartholdiErschler:imbeddings}.

Let $G=\langle A\cup B\rangle$ be a group satisfying the conditions of Theorem~\ref{th:theoremone}. Let $\zeta\in\X$ be a generic point of the Cantor set. Denote by $\Gamma_\zeta$ its \emph{orbital} graph. Its set of vertices is the $G$-orbit of $\zeta$. For every vertex $\eta$ and every generator $s\in A\cup B$ we have an edge connecting $\eta$ with $s(\eta)$, labeled by $s$. Since the generators $s\in A\cup B$ act one each point either trivially or as one of the homeomorphisms $a, b$, the graph $\Gamma_\zeta$ is just a ``decorated'' version of the orbital graph of $\zeta$ for the dihedral group $D_\infty=\langle a, b\rangle$. The latter is a bi-infinite chain, whose edges are alternatingly labeled by $a$ and $b$. The orbital graph $\Gamma_\zeta$ is obtained from it by replacing every edge labeled by $a$ or $b$ by a collection of edges labeled by some elements of $A$ or $B$, respectively, and adding loops labeled by some elements of $A\cup B$. Therefore, the graphs $\Gamma_\zeta$ are naturally represented by bi-infinite sequences $w_\zeta=\ldots x_{-1}x_0x_1\ldots$ over some finite alphabet.

Minimality of the action implies that the graphs $\Gamma_\zeta$ (equivalently, the sequences $w_\zeta$) are \emph{repetitive}. For every finite subgraph $\Sigma$ of $\Gamma_\zeta$ there exists $R_\Sigma\in\mathbb{N}$ such that for every vertex $\eta$ of $\Gamma_\zeta$ there exists an isomorphic copy (as a labeled graph) of $\Sigma$  on distance not more than $R_\Sigma$ from $\eta$ in $\Gamma_\zeta$.

We say that $G$ has \emph{linearly repetitive orbits} if there exists a constant $C$ such that $R_\Sigma$ is bounded from above by $C$ times the diameter of $\Sigma$.
We prove the following (see Theorem~\ref{th:growth}).

\begin{theorem}
\label{th:two}
Let $\langle A\cup B\rangle$ be a group satisfying the conditions of Theorem~\ref{th:theoremone}. If it has linearly repetitive orbits, then it is of intermediate growth. If the action of the dihedral group $D_\infty$ is expansive, then the corresponding group $\alt(\langle A\cup B\rangle, \X)$ is finitely generated, simple, periodic, and has intermediate growth.
\end{theorem}

Thus the answer to Problem~\ref{prob:grigorchuk} is positive.
For properties of linearly repetitive (also called \emph{linearly recurrent}) dynamical systems and quasi-crystals, and applications to spectral theory of Schr\"odinger operators, see~\cite[Chapters~6 and~9]{kellendonLenz:aperiodicorder} and~\cite{damaniklenz}.

Linear repetitivity of infinite sequences is a stronger condition than their \emph{linear complexity}. An infinite sequence $w$ has linear complexity if the number different subwords of length $n$ of the sequence $w$ is bounded from above by $Cn$ for some $n$. If a group $G$ generated by a fragmentation of a minimal action of the dihedral group has orbital graphs of linear complexity, then $G$ is Liouville, by a theorem of N.~Matte~Bon~\cite{mattebon:liouville} (which is applicable to a more general type of groups and with a weaker condition on the sequences). The Liouville condition (absence of non-constant bounded harmonic functions) is stronger than amenability, but weaker than subexponential growth.

Our method of proving periodicity and intermediate growth is substantially different from the original proofs of periodicity and intermediate growth of the Grigorchuk group, since we can not use an action on a rooted tree. All previous proofs of intermediate growth of a group used ``length reduction'' of automorphisms of rooted trees, as in the original paper of R.~Grigorchuk, or used intermediate growth of the Grigorchuk groups.

We study how points travel inside the orbital graphs $\Gamma_\zeta$ under the action of positive powers of one element (to prove periodicity) or under the action of a long product of generators (to prove intermediate growth). In both cases we use one-dimensional structure of the orbit, i.e., the fact that a trajectory starting in one vertex and ending in another has to pass through all the vertices between them. Small neighborhoods of the special point $\xi$ from Theorem~\ref{th:theoremone} act as ``reflectors'': trajectories approaching them often ``bounce'' and change their direction. This is used in Theorem~\ref{th:theoremone} to prove that a sequence $g^k(\zeta)$, $k\ge 1$, must eventually come back to $\zeta$, thus getting periodicity of $g$. Similar idea shows that in the case when $\Gamma_\zeta$ is linearly repetitive, the trajectory $\zeta, g_1(\zeta), g_2g_1(\zeta), \ldots, g_n\cdots g_2g_1(\zeta)$ of a vertex $\zeta$ under a long product $g_n\cdots g_2g_1$ of generators of $G$ tends to change its direction often, so that it rarely goes far away. This gives a subexponential estimate  of the form $C_1\exp\left(\frac{n}{\exp(C_2\sqrt{\log n})}\right)$ on the total number of elements $g_n\cdots g_2g_1\in G$, thus proving Theorem~\ref{th:two}.

It is interesting to note that the proof of the main result of~\cite{juschenkomonod}, as analyzed in~\cite{YNS} and~\cite{JMMS:extensive}, is also using a similar idea: trajectories of a random walk on the orbital graph $\Gamma_\zeta$ eventually return back to $\zeta$ with probability one. This, together with the one-dimensional structure of the graph implies amenability of the group. Here we do not need the ``reflectors'' produced by a special point of $\X$, since we need only a probabilistic result in this case.

Orbital graphs of the
Grigorchuk groups were studied in great detail by Y.~Vorobets in~\cite{vorob:schreiergraphs};
an important part of our construction is based on his results.

Section~\ref{s:prelim} contains preliminary general facts on groups acting on topological spaces, orbital graphs, graphs of germs, and minimal actions of the dihedral group. In Section~\ref{s:fragmentations} we define fragmentations of dihedral groups, and study their orbital graphs and graphs of germs. Section~\ref{s:periodicity} contains the proof of Theorem~\ref{th:theoremone}. Theorem~\ref{th:two} is proved in Section~\ref{s:growth}. The last Section~\ref{s:fibonacci} describes in detail one particular example: a fragmentation of the substitutional Fibonacci shift. We include it here to give an explicit example of a finitely generated simple periodic group of intermediate growth, without relying on somewhat indirect proof of finite generation of $\alt(G, \X)$ in~\cite{nek:fullgr}.

The author is very grateful to Laurent Bartholdi, Yves de Cornulier, Rostislav Grigorchuk, Kate Juschenko, Nicolas Matte~Bon, Mark Sapir, and
Said Sidki for remarks and suggestions.

\section{Preliminaries on group actions}
\label{s:prelim}
We use left actions, so in a product $a_1a_2$ the transformation $a_2$ is performed before $a_1$. We denote the identity transformation and the identity element of a group by $\unit$. The symmetric and the alternating group acting on a set $A$ is denoted $\symm(A)$ and $\alt(A)$, respectively.

For a finite alphabet $\alb$, we denote by $\xo$ the space of infinite one-sided sequences $x_1x_2\ldots$ of elements of $\alb$, and by $\alb^{\Z}$ the space of two-sided sequences $\ldots x_{-2}x_{-1}\;.\;x_0x_1\ldots$. Both spaces are endowed with the direct product topology, where $\alb$ is discrete. We denote by $\xs$ the set of all finite words over the alphabet $\alb$, i.e., the free monoid generated by $\alb$.

\subsection{Graphs of actions}
\label{ss:graphsofactions}

All graphs in this section are oriented, loops and multiple edges are
allowed. Their edges are labeled. Distances
between vertices in such graphs are measured ignoring the
orientation. Similarly, connectedness and connected components are also
defined ignoring the orientation. Isomorphisms must preserve
orientation and labeling. A graph is called \emph{rooted} if one
vertex, called the \emph{root}, is marked. Every morphism of rooted
graphs must map the root to the root.

We denote a ball of radius $r$ with center in a vertex $v$ of a graph
$\Gamma$ by $B_v(r)$. It is considered to be a rooted graph (with root the 
$v$). Its set of edges is the set of all edges of $\Gamma$ connecting
the vertices of $B_v(r)$. The orientation and labeling are inherited from $\Gamma$.

Let $G$ be a group generated by a finite set $S$ and acting by
homeomorphisms on a compact metrizable space $\X$.

For $\zeta\in\X$,
the \emph{orbital graph} $\Gamma_\zeta$ is the graph with the set of
vertices equal to the orbit $G\zeta$ of $\zeta$, in which for every
$\eta\in G\zeta$ and every $s\in S$ there is an arrow from $\eta$ to
$s(\eta)$ labeled by $s$.

The graph $\Gamma_\zeta$ is naturally isomorphic to the \emph{Schreier
  graph} of the group $G$ modulo the stabilizer $G_\zeta$. The
Schreier graph of $G$ modulo a subgroup $H$ is, by definition, the graph with the
set of vertices equal to the set of cosets $gH$, $g\in G$, in which for
every coset $gH$ and every generator $s\in S$ there is an arrow
from $gH$ to $sgH$ labeled by $s$.

Denote by $G_{(\zeta)}$ the subgroup of elements of $G$ acting trivially on a neighborhood of $\zeta$, i.e., the subgroup of all elements
$g\in G$ such that $\zeta$ is an interior point of the set of fixed
points of $g$. The
\emph{graph of germs} $\widetilde\Gamma_\zeta$ is the Schreier graph of $G$
modulo $G_{(\zeta)}$. Note that $G_{(\zeta)}$ is a normal subgroup of
$G_\zeta$, hence the map $hG_{(\zeta)}\mapsto hG_\zeta$ induces a
Galois covering of graphs $\widetilde\Gamma_\zeta\arr\Gamma_\zeta$ with
the group of deck transformations $G_\zeta/G_{(\zeta)}$. We call $G_\zeta/G_{(\zeta)}$ the \emph{group of germs} of the point $\zeta$.

The vertices of $\widetilde\Gamma_\zeta$ are identified with \emph{germs}
of elements of $G$ at $\zeta$. Here a germ is an
equivalence class of a pair $(g, \zeta)$, where two pairs $(g_1,
\zeta)$ and $(g_2, \zeta)$ are equivalent if there exists a
neighborhood $U$ of $\zeta$ such that $g_1|_U=g_2|_U$.

\begin{defi}
A point $\zeta\in\X$ is said to be \emph{$G$-regular} if its group of germs is trivial, i.e., if every element $g\in G$ fixing $\zeta$
fixes pointwise a neighborhood of $\zeta$. If $\zeta$ is not $G$-regular, then we say that it is \emph{singular}.
\end{defi}

Note that for every $g\in G$ the set of points $\zeta\in\X$ such that
$g(\zeta)=\zeta$ but $g\notin G_{(\zeta)}$ is equal to the boundary of
the set of fixed points of $g$. It follows that this set is closed
and nowhere dense. Consequently, if $G$ is countable (in particular, if
$G$ is finitely generated), then the set of $G$-regular points is co-meager
(residual).

Note also that $gG_\zeta g^{-1}=G_{g(\zeta)}$ and
$gG_{(\zeta)}g^{-1}=G_{(g(\zeta))}$ for all $\zeta\in\X$ and $g\in G$,
which implies that the set of $G$-regular points is $G$-invariant.

Depending on  the separation conditions for the elements of the group of germs $G_\zeta/G_{(\zeta)}$ with respect to the natural topology on the groupoid of germs,  singular points can be classified in the following way.

\begin{defi}
\label{def:singularities}
Suppose that $\zeta\in\X$ is a singular point.

We say that $\zeta$ is a \emph{Hausdorff singularity} if for every $g\in G_\xi\setminus G_{(\xi)}$ the interior of the set of fixed points of $g$ does not accumulate on $\xi$. Otherwise, $\zeta$ is a \emph{non-Hausdorff singularity}.

We say that $\zeta$ is a \emph{purely non-Hausdorff singularity} if for every $g\in G_\xi$ the interior of the set of fixed points of $g$ accumulates on $\xi$.
\end{defi}

Let $(\Gamma_1, v_1)$, $(\Gamma_2, v_2)$ be connected rooted labeled graphs, where $v_i$ are the roots. Define the distance $d((\Gamma_1, v_1), (\Gamma_2, v_2))$ between them as
$2^{-(R+1)}$, where $R$ is the maximal integer
such that the balls $B_{v_1}(R)\subset\Gamma_1$ and
$B_{v_2}(R)\subset\Gamma_2$ of radius $R$ with centers in $v_1$ and
$v_2$ are isomorphic as rooted graphs. This metric defines a natural topology on
the space $\mathcal{G}$
of all isomorphism classes of connected oriented rooted labeled graphs. If we fix a finite set of
labels, then the space of all such graphs is compact.

It is easy to see that the map $\zeta\mapsto(\Gamma_\zeta, \zeta)$ from $\X$ to the space of labeled rooted graphs is continuous at regular points. The map $\zeta\mapsto(\widetilde\Gamma_\zeta, \zeta)$ is continuous at regular points and at Hausdorff singularities.

\begin{defi}
The action of $G$ on $\X$ is said to be \emph{minimal} if all
$G$-orbits are dense in $\X$.
\end{defi}

\begin{proposition}
\label{prop:regularcontainment}
Suppose that the action of $G$ on $\X$ is minimal. Let $\zeta\in\X$ be a
$G$-regular point. Then for every
ball $B_\zeta(r)$ of $\Gamma_\zeta$ there exists a number $R(r)>0$ such that
for every $\eta\in\X$ there exists a vertex $\eta'$ of $\Gamma_\eta$ such
that $d(\eta, \eta')\le R(r)$ and the rooted balls $B_\zeta(r)$ and $B_{\eta'}(r)$ are isomorphic.
\end{proposition}

\begin{proof}
The ball $B_\zeta(r)$ can be described by a finite system of equations and
inequalities of the form $g_1(\zeta)=g_2(\zeta)$ or $g_1(\zeta)\ne g_2(\zeta)$, for
pairs of elements $g_1, g_2\in G$ of length at most $r$. Such an equality or inequality is valid for all points of some
neighborhood of $\zeta$. It follows that there exists a neighborhood $N$
of $\zeta$ such that for every $\eta'\in N$ the balls $B_\zeta(r)$ and $B_{\eta'}(r)$ of
the corresponding orbital graphs are isomorphic as rooted labeled graphs.

For every point $\eta\in\X$ there exists an element $g\in G$ such that
$g(\eta)\in N$. The set of sets of the form $g^{-1}(N)$ cover $\X$,
and, by compactness, there exists a finite subcover $g_1^{-1}(N),
g_2^{-1}(N), \ldots, g_n^{-1}(N)$. Let $R$ be the maximal length of
the elements $g_i$ with respect to the generating set $S$. Then for every
$\eta\in\X$ there exists $g_i$ such that $g_i(\eta)\in N$, and hence
the balls $B_\zeta(r)$ and $B_{\eta'}(r)$ are isomorphic for
$\eta'=g_i(\eta)$. Distance from $\eta$ to $\eta'$ is not more than $R$. Since the number of isomorphism classes of balls of radius $r$ in orbital graphs of $G$ is finite, we can find an estimate $R(r)$ independent of $\zeta$.
\end{proof}

\begin{defi}
\label{def:linearlirecurrent}
We say that the action of $G$ on $\X$ is \emph{linearly repetitive} if there exists $K>1$ such that the function $R(r)$ from Proposition~\ref{prop:regularcontainment} satisfies $R(r)<Kr$ for all $r>1$.
\end{defi}

\subsection{Topological full groups}
\label{ss:full}

Let $G$ be a group acting on a Cantor set $\X$. The \emph{topological
  full group} $\full(G, \X)$ of the action is the group of all
homeomorphisms $h:\X\arr\X$ such that for every $\zeta\in\X$ there
exists a neighborhood $U$ of $\zeta$ and an element $g\in G$ such that
$h|_U=g|_U$.
Topological full groups were introduced in~\cite{gior:full} (see also~\cite{krieger:homgroups}, where they appeared earlier). See the
papers~\cite{matui:fullI,matui:expgrowth,matui:fullonesided,matui:etale} and the survey~\cite{cornulier:pleinstopologiques}
for various properties of topological full groups of group actions and
\'etale groupoids.

Let $U\subset\X$ be a non-empty clopen set, and let $g_1, g_2, \ldots, g_n\in G$
be such that the sets $U_1=g_1(U), U_2=g_2(U), \ldots, U_n=g_n(U)$ are pairwise
disjoint. Then for every permutation $\alpha\in\symm_n=\symm(\{1, 2, \ldots, n\})$ we get the
corresponding element $h_\alpha$ of the topological full group acting by the
rule:
\[h_\alpha(\zeta) = \left\{\begin{array}{rl} g_jg_i^{-1}(\zeta) &
    \text{if $\zeta\in U_i$ and $\alpha(i)=j$;}\\
\zeta & \text{if $\zeta\notin\bigcup_{i=1}^n
  U_i$.}\end{array}\right.\]
The map $\alpha\arr h_\alpha$ is a monomorphism from $\symm_n$ to
$\full(G)$. Denote by $\alt(G, \X)$ the subgroup
generated by the images of the alternating subgroups $\alt_n<\symm_n$ for
all such monomorphisms.

The following is proved in~\cite{nek:fullgr}.

\begin{theorem}
\label{th:fullgrnek}
If the action of $G$ on $\X$ is minimal, then $\alt(G, \X)$ is
simple and is contained in every non-trivial normal subgroup of $\full(G, \X)$. If the action of $G$ on $\X$ is expansive and has infinite
orbits, then $\alt(G, \X)$ is finitely generated.
\end{theorem}

\begin{defi}
An action of $G$ on $\X$ is said to be \emph{expansive} if there exists
$\delta>0$ such that $d(g(\zeta_1), g(\zeta_2))<\delta$ for all $g\in
G$ implies $\zeta_1=\zeta_2$ (where $d$ is a metric on $\X$ compatible with the topology).
\end{defi}

 An action $(G, \X)$ on a Cantor set is expansive if
and only if there exists a $G$-equivariant homeomorphism from $\X$ to
a closed $G$-invariant subset of $A^G$ for some finite alphabet $A$.

\subsection{Minimal actions of the dihedral group}

When a set of generators $S$ of a group $G$ consists of elements of
order two, then we will consider the orbital graphs and graphs of germs
as non-oriented, so that an edge connecting two vertices $v_1$ and
$v_2$ labeled by $s\in S$ replaces two arrows labeled by $s$: one from
$v_1$ to $v_2$ and one from $v_2$ to $v_1$ (if the edge is not a loop).

Let $a, b$ be homeomorphisms of period two of a Cantor set $\X$ such
that the dihedral group $\langle a, b\rangle$ acts minimally on $\X$.

\begin{lemma}
The orbital graphs of $\langle a, b\rangle$ are either one-ended or two-ended  infinite chains. The graphs of germs are two-ended infinite chains.
\end{lemma}

\begin{proof}
The Schreier
graphs of the infinite dihedral group $D_\infty$ are either infinite
chains (one-ended or two-ended), or finite chains, or finite
cycles. The latter two cases are impossible, since then we have a
finite orbit, which contradicts with minimality.

Suppose that a graph of germs is a one-ended infinite chain. Then the
endpoint of the chain is a fixed point of one of the generators, and there are no other fixed points of the generators in this orbit. Since
this is a graph of germs, it follows that the generator fixes this
point together with every point of a neighborhood. But then, by minimality, there are
other points of the orbit where the germ of the generator is trivial, which is a contradiction.
\end{proof}

\begin{corollary}
\label{cor:stabaa}
If the stabilizer $\langle a, b\rangle_\xi$ is non-trivial, then there exists a unique point $\xi'$ in the orbit of $\xi$ such that $\langle a, b\rangle_{\xi'}$ is equal either to $\langle a\rangle$ or to $\langle b\rangle$.
\end{corollary}

Let us show how minimality and expansivity conditions for $D_\infty$- and $\Z$-actions are related.

\begin{proposition}
Let $a$ and $b$ be homeomorphisms of period two of a Cantor set $\X$.
If the action the dihedral group $\langle a, b\rangle$ is minimal, then either the action of $\langle ab\rangle$ is minimal, or $\X$ is split into a disjoint union of two clopen $\langle ab\rangle$-invariant sets $S_1, S_2$ such that the action of $\langle ab\rangle$ on  each of these sets is minimal, and $a(S_1)=b(S_1)=S_2$, $a(S_2)=b(S_2)=S_1$.
\end{proposition}

In particular, if the action of $D_\infty$ is non-free (has non-trivial stabilizers of some points), then the $D_\infty$-minimality is equivalent to the
$\Z$-minimality.

\begin{proof}
Suppose that the $\langle a, b\rangle$-action is minimal.
If $A\subset\X$ is a closed non-empty $\langle ab\rangle$-invariant set, then $a(A)$ is also a closed $\langle ab\rangle$-invariant set (since $(ab)a(A)=a(ba)A=a(ab)^{-1}(A)=a(A)$). It follows that $a(A)\cap A$ and $a(A)\cup A$ are closed and $\langle a, b\rangle$-invariant.
Consequently, $a(A)\cup A=\X$, and
either $a(A)\cap A=\X$, or $a(A)\cap A=\emptyset$, which finishes the proof.
\end{proof}

\begin{proposition}
\label{pr:Zsubshift}
Let $a$ and $b$ be homeomorphisms of period two of a Cantor set $\X$.
Suppose that they generate an
expansive action of $D_\infty$. Then there exists
a finite alphabet $A$, a permutation $\iota:A\arr A$ such that $\iota^2=\unit$, and a $\Z$-subshift
$\mathcal{S}\subset A^{\Z}$ such that there exists a homeomorphism
$\X\arr\mathcal{S}$ conjugating the action of the generators $a$
and $b$ with the homeomorphisms of $\mathcal{S}$ given by the
formulas:
\[a(w)(n)=\iota(w(-n)),\qquad b(w)(n)=\iota(w(1-n))\]
for every $w\in\mathcal{S}$ and $n\in\Z$.
\end{proposition}

Recall, that a \emph{subshift} is a closed $\Z$-invariant subset of $A^{\Z}$.

\begin{proof}
There exists a partition $\mathcal{U}=\{U_1,
U_2, \ldots, U_n\}$ of $\X$ into clopen sets such that every
point $\zeta\in\X$ is uniquely determined by its \emph{itinerary},
which is defined as the map $I_\zeta:D_\infty\arr\mathcal{U}$ given by the
condition $I_\zeta(g)\ni g(\zeta)$. We may assume that $\mathcal{U}$ is
$a$-invariant, i.e., that for every $U\in\mathcal{U}$ the set $a(U)$
belongs to $\mathcal{U}$. Otherwise, we can replace $\mathcal{U}$ by
the partition induced by $\mathcal{U}$ and $a(\mathcal{U})$: two
points $\zeta_1, \zeta_2$ belong to one piece of the induced partition
if and only if they belong to one piece of $\mathcal{U}$ and to one
piece of $a(\mathcal{U})$.

Then for every $\zeta\in\X$ and $g\in D_\infty$ we have
$I_\zeta(g)=a(I_\zeta(ag))$,
 so that $I_\zeta$, and hence $\zeta$, are
uniquely determined by the sequence $I_\zeta((ab)^n)$, $n\in\Z$. Let
us denote $J_\zeta(n)=I_\zeta((ab)^n)$.

The set of sequences of the form $J_\zeta(n)$ is obviously
a closed shift-invariant subset of the full shift $\mathcal{U}^{\Z}$.

Let us describe the action of $a$ and $b$ on the sequences
$J_\zeta(n)$. We have
$J_{a(\zeta)}(n)=I_\zeta((ab)^na)=I_\zeta(a(ba)^n)=a(I_\zeta((ba)^n))=a(J_\zeta(-n))$
and
$J_{b(\zeta)}(n)=I_\zeta((ab)^nb)=I_\zeta(a(ba)^{n-1})=a(I_\zeta((ba)^{n-1})=a(J_\zeta(1-n))$.
We can therefore define the permutation $\iota$ of $\mathcal{U}$ equal to the action of $a$ on $\mathcal{U}$.
\end{proof}

If the permutation $\iota$ in Proposition~\ref{pr:Zsubshift} is trivial, then the transformations $a$ and $b$ are ``central symmetries'' of the infinite sequences from the subshift $\mathcal{A}$. The transformation $a$ flips the sequences around the zeroth letter, while $b$ flips them around the space between the zeroth and first letters. The subshift $\mathcal{S}$ has to be invariant under $a$ (and then it will be invariant under $b$). Such subshifts are called \emph{palindromic}. A minimal subshift is palindromic if and only if a sequence $w\in\mathcal{S}$ contains arbitrary long palindromes as subwords. So, every palindromic minimal subshift is associated with a natural minimal expansive action of $D_\infty$. The general case is not far from palindromic subshifts, since we can replace letters $x\in A$ such that $\iota(x)\ne x$ by words of length two: $x=x_-x_+$, $\iota(x)=x_+x_-$, so that $a$ is again a central symmetry.

\begin{example}
\label{ex:ThueMorse}
Let $\tau$ be the substitution (i.e., an endomorphism of the free
monoid $\{\zero, \one\}^*$) given by:
\[\tau:\zero\mapsto \zero\one,\qquad \one\mapsto \one\zero.\]
The words $\tau^n(\zero)$ converge to an infinite sequence
$\zero\one\one\zero\one\zero\zero\one\one\zero\zero\one\zero\one\one\zero\ldots$
called the \emph{Thue-Morse sequence}. Let $\mathcal{S}$ be the set of all
bi-infinite sequences $w=\ldots x_{-1}x_0x_1\ldots$ such that every
subword of $w$ is a subword of $\lim_{n\to\infty}\tau^n(\zero)$. It is
known that $\mathcal{S}$ is a minimal subshift (see~\cite[Example~10.9.3]{allouchshallit:sequences}).

Note that the words $\tau^2(\zero)$ and $\tau^2(\one)$ are palindromes:
\[\tau^2(\zero)=\zero\one\one\zero,\qquad\tau^2(\one)=
\one\zero\zero\one.\]
It follows by induction that $\tau^{2n}(\zero)$ and $\tau^{2n}(\one)$ are
palindromes for all $n\ge 1$. Consequently, the shift $\mathcal{S}$ is
palindromic, and the central symmetries $a$ and $b$ around the position number 0 and the space between positions number 0 and 1 generate a minimal expansive action of $D_\infty$.
\end{example}

\begin{example}
Consider the alphabet $\alb=\{\one, \one^*, \two, \two^*\}$, the involution $\iota:x\leftrightarrow x^*$, $x\in\{\one, \two\}$, and the substitution
\[\tau:\one\mapsto\two,\quad\one^*\mapsto\two^*,%
\quad\two\mapsto\one^*\two^*,\quad\two^*\mapsto\two\one.\]
We have $\iota\circ\tau=\tau\circ\iota$ on $\xs$, where $\iota(x_1x_2\ldots x_n)=\iota(x_n)\iota(x_{n-1})\ldots\iota(x_1)$.

Let $\mathcal{S}\subset\alb^{\Z}$ be the subshift generated by $\tau$. We have $\tau^3(\one)=\two^*\two\one$. Since $\iota$ commutes with $\tau$, all words $\tau^n(\two^*\two)$ are $\iota$-invariant:
\[\two^*|\two,\quad\two\one|\one^*\two^*,\quad \one^*\two^*\two|\two^*\two\one,\quad\ldots\]

It follows that the shift $\mathcal{S}$ is invariant under the transformations $a$ and $b$ defined as in Proposition~\ref{pr:Zsubshift}.
\end{example}

\subsection{Odometer actions}
\label{ss:odometer}

In some sense the opposite condition to expansiveness is \emph{residual
  finiteness} of the action. We say that an action of a group $G$ on a
Cantor set $\X$ is residually finite if the $G$-orbit of every clopen
subset of $\X$ is finite. An action is residually finite if and only
if there exists a homeomorphism $\Phi:\X\arr\partial T$ of $\X$ with
the boundary of a locally finite rooted tree $T$ and an action of $G$ on $T$ by
automorphisms such that $\Phi$ is $G$-equivariant (with respect to the
action of $G$ on $\partial T$ induced by the action on $T$),
see~\cite[Proposition~6.4]{grineksu_en}.

Every minimal residually finite action of $\Z$ on a Cantor set is topologically conjugate to an \emph{odometer}, i.e., the transformation $\alpha:x\mapsto x+1$ on the projective limit $\overline{\Z}$ of a sequence $\Z/(d_1d_2\ldots d_n)\Z$ of finite cyclic groups.

\begin{proposition}
Consider a minimal residually finite action of the dihedral group $D_\infty$ on a Cantor set $\X$. Suppose that the restriction of the action to the infinite cyclic subgroup of $D_\infty$ is minimal. Then there exists a homeomorphism of $\X$ with a projective limit $\overline{\Z}$ of finite cyclic groups conjugating the action of $D_\infty$ with the action generated by the homeomorphisms
\[a(x)=1-x,\qquad b(x)=-x.\]
\end{proposition}

\begin{proof}
The homeomorphism $\alpha$ generating the infinite cyclic subgroup of $D_\infty$ is an odometer, i.e., is conjugate to the action of $x\mapsto x+1$ on some projective limit $\overline{\Z}$ of finite cyclic groups.

Let $a$ and $b$ be the generators of $D_\infty$.
We have $b\alpha b=\alpha^{-1}$. If $b_1, b_2\in D_\infty$ are such that $b_i\alpha b_i=\alpha^{-1}$, then $b_1b_2$ commutes with $\alpha$. By continuity and minimality of the action of $\alpha$ on $\overline{\Z}$, the homeomorphism $b_1b_2$ commutes with every transformation of the form $x\mapsto x+h$ for $h\in\overline{\Z}$. It follows that $b_1b_2(h)=b_1b_2(h+0)=h+b_1b_2(0)$ for every $h$, i.e., that $b_1b_2$ is of the form $x\mapsto x+g$ for some $g\in\overline{\Z}$.

The transformation $b_0(x)=-x$ satisfies $b_0\alpha b_0=\alpha^{-1}$. It follows that every involution $b$ such that $b\alpha b=\alpha^{-1}$ is of the form $b(x)=-x+g$ for some $g\in\overline{\Z}$. Then $a=\alpha b$ is given by $a(x)=-x+g+1$.

If all cyclic groups $d_1d_2\cdots d_n\Z$ in the projective limit have odd order, then the equation $2x=g$ has a solution in $\overline{\Z}$ for every $g\in\overline{\Z}$. Otherwise, either the equation $2x=g$, or the equation $2x=g+1$ has a solution. It follows that in the odd case both involutions $a, b$ have a fixed point, while in the even case exactly one of them has a fixed point.

Let us assume that $b$ has a fixed point $\xi\in\overline{\Z}$. Then, conjugating everything by the shift $x\mapsto x-\xi$, we may assume that $\xi=0$. Then $a:x\mapsto -x+1$ and $b:x\mapsto -x$.
\end{proof}

For example, if $\X$ is the ring of diadic integers, i.e., the projective limit of the cyclic groups $\Z/2^n\Z$, then the corresponding action of $a:x\mapsto 1-x$ and $b:x\mapsto -x$ is conjugate to the following action on the space $\{\zero, \one\}^\omega$ of right-infinite binary sequences:
\begin{alignat*}{2}
a(\zero w) &=\one w, &\qquad a(\one w)&=\zero w,\\
b(\zero w) &=\zero a(w), &\qquad b(\one w)&=\one b(w).
\end{alignat*}
Figure~\ref{fig:gray} shows the corresponding action on the binary rooted tree.

\begin{figure}
\centering
\includegraphics{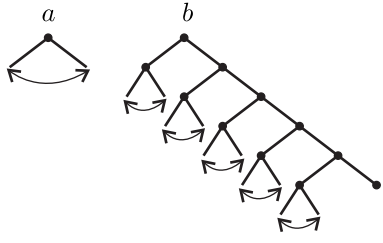}
\caption{The Gray code action of $D_\infty$}
\label{fig:gray}
\end{figure}

This particular realization corresponds to the Gray code (see~\cite[Chapter~1]{wilf:algorithms}). It is also the standard self-similar action of the iterated monodromy group of the Chebyshev polynomial $T_2=2x^2-1$. The iterated monodromy group of the degree $d$ Chebyshev polynomial is conjugate to the natural action of $D_\infty$ on the ring $\lim_{\leftarrow}\Z/d^n\Z$ of $d$-adic integers. For their standard self-similar actions, see~\cite[Proposition~6.12.6]{nek:book}.

\section{Fragmentation of dihedral groups}
\label{s:fragmentations}

\begin{defi}
Let $a$ be a homeomorphism of period two of a Cantor set $\X$. A \emph{fragmentation} of $a$ is a finite group $A$ of homeomorphisms of $\X$ such that for every $h\in A$ and $\zeta\in\X$ either $h(\zeta)=\zeta$ or $h(\zeta)=a(\zeta)$, and for every $\zeta\in\X$ there exists $h\in A$ such that $h(\zeta)=a(\zeta)$.
\end{defi}

Note that if $h$ is an element of a fragmentation $A$, then the sets $E_{h, \unit}=\{\zeta\in\X\;:\;h(\zeta)=\zeta\}$ and $E_{h, a}=\{\zeta\in\X\;:\;h(\zeta)=a(\zeta)\}$ are closed, $a$-invariant, their intersection is a set of fixed point of $a$, and we have $E_{h, \unit}\cup E_{h, a}=\X$. If the set of fixed points of $a$ has empty interior (e.g., if it is an element of a generating set of a minimal action of $D_\infty$), then the interiors of $E_{h, \unit}$ and $E_{h, a}$ are disjoint.

Suppose that the set of fixed points of $a$ has empty interior.
Choose for every $h\in A$ a set $Q_h$ either equal to the interior of $E_{h, \unit}$ or to the interior of $E_{h, a}$, and consider their intersection. Let $\mathcal{P}$ be the set of all intersections that can be obtained this way. The set $\mathcal{P}$ has the following properties:
\begin{enumerate}
\item $\mathcal{P}$ is finite;
\item the elements of $\mathcal{P}$ are $a$-invariant, open, and pairwise disjoint;
\item $\bigcup_{P\in\mathcal{P}}\overline P=\X$;
\item for all $P_1, P_2\in\mathcal{P}$ such that $P_1\ne P_2$ the set $\overline{P_1}\cap\overline{P_2}$ consists of fixed points of $a$;
\item for every $P\in\mathcal{P}$ and $h\in A$, either $P\subset E_{h, \unit}$, or $P\subset E_{h, a}$.
\end{enumerate}

We call the elements of $\mathcal{P}$ the \emph{pieces} of the fragmentation $A$. Every piece $P\in\mathcal{P}$ defines an epimorphism $\pi_P:A\arr\Z/2\Z$ by the rule:
\[\pi_P(h)=\left\{\begin{array}{rl} 0 & \text{if $P\subset E_{h, \unit}$,}\\
1  & \text{if $P\subset E_{h, a}$.}\end{array}\right.\]
In other words, $\pi_P(h)=1$ if $h$ acts on $P$ as $a$, and $\pi_P(h)=0$ if $h$ acts on $P$ as the identity. The map $(\pi_P)_{P\in\mathcal{P}}$ defines an embedding of $A$ into $(\Z/2\Z)^{\mathcal{P}}$. The group $(\Z/2\Z)^{\mathcal{P}}$ is, therefore, the maximal fragmentation of $a$ with the set of pieces $\mathcal{P}$. Any subgroup $A\le(\Z/2\Z)^{\mathcal{P}}$ such that all homomorphisms $\pi_P:A\arr\Z/2\Z$ are surjective is a fragmentation with the set of pieces $\mathcal{P}$.

Our main subject are fragmentations with purely non-Hausdorff singularities, see Definition~\ref{def:singularities}. Every non-free minimal action of $D_\infty$ can be fragmented so that we get a purely non-Hausdorff singularity in the following way.

\begin{lemma}
\label{lem:partition} Suppose that $a$ has a fixed point $\xi$. Then
for every $n\ge 1$ there exists a partition of $\X\setminus\{\xi\}$ into
a disjoint union of open $a$-invariant subsets $P_1, P_2, \ldots, P_n$
such that each set $P_i$ accumulates on $\xi$.
\end{lemma}

\begin{proof}
Let $U_k, k\ge 0$, be
a descending sequence of clopen neighborhoods of $\xi$ such that
$U_0=\X$ and
$\bigcap_{k\ge 0}U_k=\{\xi\}$. 

Then $V_k=U_k\cap a(U_k)$ is a
descending sequence of clopen $a$-invariant neighborhoods of $\xi$ such
that $\bigcap_{k\ge 1}V_k=\{\xi\}$. Remove all repetitions, so that
$V_k\ne V_{k+1}$ for every $k$.

Choose an arbitrary partition of the set of non-negative integers into
$n$ disjoint infinite subsets $I_1, I_2, \ldots, I_n$, and define
$P_i=\bigcup_{k\in I_i}V_k\setminus V_{k+1}$.
\end{proof}

Suppose that $\mathcal{P}=\{P_1, P_2, \ldots, P_n\}$ is as in Lemma~\ref{lem:partition}. Choose a subgroup $A\le(\Z/2\Z)^{\mathcal{P}}$ such that each homomorphism $\pi_{P_i}:A\arr\Z/2\Z$ is surjective, but there is no element $h\in A$ such that $\pi_{P_i}(h)=1$ for all $P_i$. Then $A$ is a fragmentation of $a$ such that $\xi$ is a purely non-Hausdorff singularity. It is always possible to choose such $A$ if $n\ge 3$. For example, for $n=3$ such a subgroup of $(\Z/2\Z)^3$ is $\{(0, 0, 0), (1, 1, 0), (1, 0, 1), (0, 1, 1)\}$. 

\begin{example}
\label{ex:Grigorchuk}
Consider the space $\{\zero, \one\}^\omega$ of right-infinite sequences $x_1x_2\ldots$ over the
binary alphabet $\{\zero, \one\}$. Consider the Gray code transformations $a$ and $b$, as defined in~\ref{ss:odometer}.

The homeomorphism $b$ has a
unique fixed point $\xi=\one\one\one\ldots$.
The sets $W_n=\underbrace{\one\one\ldots \one}_{\text{$n$ times}}\zero\{\zero,
\one\}^\omega$ of sequences starting with exactly $n$ ones form a
partition of $\{\zero, \one\}^\omega\setminus\{\xi\}$ into open
$b$-invariant subsets.

Consider the partition $P_0=\bigcup_{k=0}^\infty W_{3k}$,
$P_1=\bigcup_{k=0}^\infty W_{3k+1}$, $P_2=\bigcup_{k=0}^\infty
W_{3k+2}$ of $\{\zero, \one\}^\omega\setminus\{\xi\}$, and the subgroup
$B=\{b_1, b_2, b_3, \unit\}$, where $b_1$ acts as $b$ on $P_0\cup P_1$, $b_2$ acts as $b$ on $P_0\cup P_2$, and $b_3$ acts as $b$ on $P_1\cup P_2$. The group generated by $a$ and $B$ is the \emph{first Grigorchuk group}, introduced in~\cite{grigorchuk:80_en}. Its generators $a,
b_1, b_2, b_3$ are usually denoted $a, b, c, d$. See Figure~\ref{fig:grigorchuk} for a description of their action on the binary tree, where the boundary is naturally identified with the space $\{\zero, \one\}^\omega$.

\begin{figure}
\centering
\includegraphics{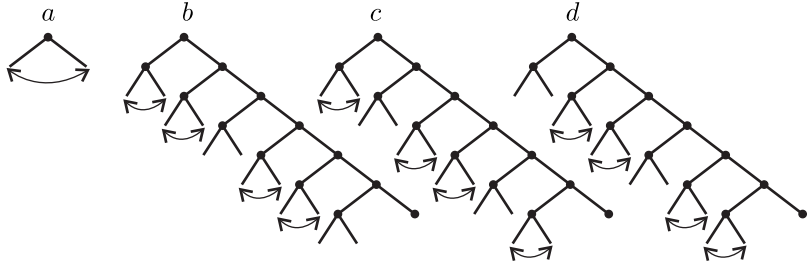}
\caption{The Grigorchuk group}
\label{fig:grigorchuk}
\end{figure}

Choosing
different sets $P_0, P_1, P_2$ equal to unions of the sets $W_k$, we
get all groups from the family of Grigorchuk groups $G_w$ studied
in~\cite{grigorchuk:growth_en}. If we choose other number of pieces in
a partition $\mathcal{P}$, then we get groups defined and studied
by Z.~\v Suni\'c in~\cite{sunic:hausdorff}.
\end{example}

\subsection{Orbital graphs of fragmented dihedral groups}
\label{ss:orbitalgraphs}

\begin{defi}
Let $a$ and $b$ be homeomorphisms of period two of a Cantor set $\X$.
A  \emph{fragmentation of the dihedral group $\langle a, b\rangle$} is the group generated by $A\cup B$, where $A$ and $B$ are fragmentations of the homeomorphisms $a$ and $b$, respectively.
\end{defi}

Let $G=\langle A\cup B\rangle$ be a fragmentation of a minimal action of a dihedral group $\langle a, b\rangle$. Denote by $\mathcal{P}_A$ and $\mathcal{P}_B$ the sets of pieces of the fragmentations $A$ and $B$.

For every $\zeta\in\X$ and $a'\in A$ we have $a'(\zeta)=a(\zeta)$ or
$a'(\zeta)=\zeta$. For every $\zeta\in\X$ there exists $a'\in A$ such
that $a'(\zeta)=a(\zeta)$. The same is true for $B$ and $b$. It follows that the orbital graphs of
$G$ are just ``decorated'' versions of the orbital graphs of
$D_\infty=\langle a, b\rangle$. 

Namely, if $\zeta_1, \zeta_2$ are two different vertices of an orbital graph of $\langle a, b\rangle$ connected by an edge labeled by $a$, then $\zeta_1, \zeta_2$ belong to one piece $P\in\mathcal{P}_A$. These vertices are connected in the orbital graph of $G$ by edges labeled by all elements $h\in A$ such that $\pi_P(h)=1$. We will sometimes represent such a multiple edge by $\edge{P}$. The analogous statement is true for the edges labeled by $b$. Thus, we replace the labels $a$ and $b$ of the orbital graph of the dihedral group by pieces of the respective fragmentation.
Note that all loops of an orbital graph can be reconstructed from the other edges.

See Figure~\ref{fig:grgraphs} where an orbital graph of
the Grigorchuk group and the corresponding orbital graph of the
dihedral group are shown. Note that the edges labeled by $b$ are replaced by multiple edges labeled by $\{b_1, b_2\}$, $\{b_3, b_1\}$, or $\{b_2, b_3\}$ (completed by the necessary loops). These sets of labels correspond to the pieces $P_0$, $P_1$, and $P_2$, respectively, see Example~\ref{ex:Grigorchuk}.

\begin{figure}
\centering
\includegraphics{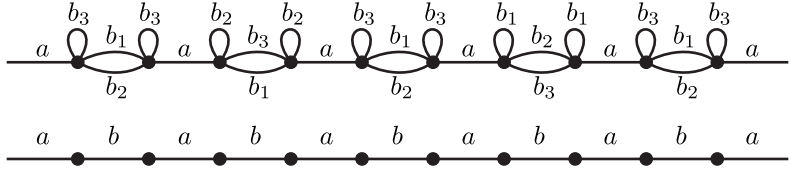}
\caption{The orbital graphs of the Grigorchuk group and the dihedral group}
\label{fig:grgraphs}
\end{figure}

A \emph{segment} $\Sigma$ is a finite connected subgraph of an orbital graph $\Gamma_\zeta$ such that if $v_1, v_2$ are adjacent vertices of $\Sigma$, then all edges of $\Gamma_\zeta$ connecting $v_1$ and $v_2$ belong to $\Sigma$. We \emph{do not}, however, include the loops of the edpoints of $\Sigma$ into the segment, for a technical reason.  

We will sometimes arbitrarily choose a direction (left/right) on a graph $\Gamma_\zeta$. Orientation of subsegments of $\Gamma_\zeta$ will be induced from the orientation of $\Gamma_\zeta$. If $\Sigma$ is an oriented segment, then we denote by $\Sigma^{-1}$ the segment with the opposite orientation. We denote by $|\Sigma|$ the length of $\Sigma$, i.e., the number of its vertices minus one.

By Corollary~\ref{cor:stabaa}, every stabilizer $\langle a, b\rangle_\zeta$ of a point $\zeta\in\X$ is conjugate to a stabilizer equal to $\langle a\rangle$ or $\langle b\rangle$. It follows that every stabilizer $G_\zeta$ of a singular point is conjugate to the stabilizer of a fixed point of $a$ or $b$.

\begin{lemma}
\label{lem:germsofxi}
Suppose that $\xi\in\X$ is a fixed point of $a$. 
The group of germs $G_\xi/G_{(\xi)}$ is naturally isomorphic to the subgroup $H_\xi$ of $A$ consisting elements $h$ fixing $\xi$ and acting non-trivially on every neighborhood of $\xi$. In other words the epimorphism $G_\xi\arr G_\xi/G_{(\xi)}$ induces an isomorphism $H_\xi\arr G_\xi/G_{(\xi)}$.
\end{lemma}

\begin{proof}
Consider a germ $(g, \xi)$. We can write $g$ as a product $b_1a_1b_2a_2\cdots b_na_n$, where $a_i\in A$, $b_i\in B$, and $a_i, b_i$ are all non-trivial except maybe for $b_1$ or $a_n$. There are no fixed points of $b$ in the orbit of $\xi$, and the only fixed point of $a$ in the orbit of $\xi$ is the point $\xi$ itself. It follows that the germs $(b_n, a_n(\xi))$, $(a_{n-1}, b_na_n(\xi))$, \ldots, $(b_1, a_1b_2a_2\cdots b_na_n(\xi))$ are equal either to germs of the identity, or to the germs of the respective elements $a$ or $b$. (We use the fact that points on the boundary of the pieces of $A$ or $B$ are fixed points of $a$ or $b$, respectively.) Consequently, the germ $(g, \xi)$ is equal to a germ of the form $(ha_n, \xi)$, where $h\in\langle a, b\rangle$. In particular, the germ of an element of the stabilizer of $\xi$ is equal to the germ of an element of $A$. This finishes the proof of the lemma.
\end{proof}

Consider the graph $\Xi$ with the set of vertices equal to the direct product of $H_\xi$ with the set of vertices of $\Gamma_\xi$. 
Two vertices $(h_1, v_1)$ and $(h_2, v_2)$ of $\Xi$ are connected by an edge 
labeled by $h\in A\cup B$ if $h_1=h_2$ and $v_1$ and $v_2$ are vertices of
$\Gamma_\xi$ connected by an edge labeled by $h$, or if $v_1=v_2=\xi$ and
$h=h_1h_2$. In other words, we take $|H_\xi|$ copies of $\Gamma_\xi$,
and then connect their roots $\xi$ by a full graph (the Cayley graph of $H_\xi$) with edges naturally labeled by the elements of $H_\xi$.

\begin{proposition}
\label{pr:graphgerms}
The graph of germs $\widetilde\Gamma_\xi$ is naturally isomorphic to
$\Xi$. The action of the group of deck transformations
$G_\xi/G_{(\xi)}\cong H_\xi$ of the
covering $\widetilde\Gamma_\xi\arr\Gamma_\xi$ coincides with the natural
action of $H_\xi$ on $\Xi$.
\end{proposition}

\begin{proof}
We know (see the proof of Lemma~\ref{lem:germsofxi})
that every germ $(g, \xi)$
is equal to a germ of the form $(g'h, \xi)$, where $g'\in\{a, ba, aba,
baba, \ldots\}$ and $h\in H_\xi$. Identify the germ $(g'h, \xi)$ with the vertex $(h,
v)\in \Xi$, where $v=g'(\xi)=g(\xi)$. It is easy to check that
this identification is an isomorphism of graphs. The statement
about the action by deck transformations also follows directly from
the description of the germs $(g, \xi)$.
\end{proof}

Let $P_1, P_2, \ldots, P_n\in\mathcal{P}_A$ be all pieces of the fragmentation $A$ that accumulate on $\xi$. Then the map $(\pi_{P_i})_{i=1}^n:H_\xi\arr (\Z/2\Z)^n$ is an isomorphic embedding.

Denote by $\Lambda_i$ the quotient of
$\widetilde\Gamma_\xi$ by the action of $\ker\pi_{P_i}$. It follows from Proposition~\ref{pr:graphgerms} that
$\Lambda_i$ is the graph obtained by
taking two copies $\{0\}\times\Gamma_\xi$ and $\{1\}\times\Gamma_\xi$ of
$\Gamma_\xi$ and connecting the endpoints $(0, \xi)$ and $(1, \xi)$ by $\edge{P_i}$.

Denote by $\lambda_i:\widetilde\Gamma_\xi\arr\Lambda_i$
the natural covering map. In terms of $\Xi$ and $\Lambda_i$, it is given by the rule
$\lambda_i(h, v)=(\pi_{P_i}(h), v)$.
See Figure~\ref{fig:lambda} where a covering $\lambda_i$ for the
Grigorchuk group is shown.

\begin{figure}
\centering
\includegraphics[width=5in]{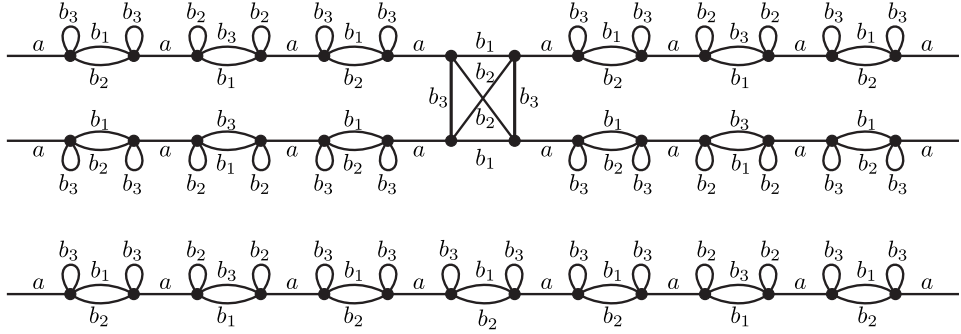}
\caption{A covering $\lambda_i$}
\label{fig:lambda}
\end{figure}

\begin{proposition}
\label{pr:Lambdai}
If $\zeta_n\in P_i$, $n\ge 1$, is a
sequence of regular points converging to $\xi$, then the rooted orbital
graphs $\Gamma_{\zeta_n}$ converge to $\Lambda_i$ in the space
of rooted labeled graphs.
\end{proposition}

For the definition of the space of rooted graphs, see Subsection~\ref{ss:graphsofactions}.

\begin{proof}
For a given positive integer $r$ consider the ball $B_\xi(r)$ of
radius $r$ in the graph of germs $\widetilde\Gamma_\xi$. It is given by a
set of equalities and inequalities of germs of the form $(g_1,
\xi)=(g_2, \xi)$ or $(g_1, \xi)\ne (g_2, \xi)$ for elements $g_1,
g_2\in G$ of length at most $r$. If $(g_1, \xi)=(g_2, \xi)$, then
$g_1(\zeta)=g_2(\zeta)$ for all $\zeta$ belonging to a
neighborhood of $\xi$. If $g_1(\xi)\ne g_2(\xi)$, then we also have
$g_1(\zeta)\ne g_2(\zeta)$ for all $\zeta$ in a neighborhood of
$\xi$. Suppose that $(g_1, \xi)\ne (g_2, \xi)$ but
$g_1(\xi)=g_2(\xi)$. Then $(g_1, \xi)=(gh_1, \xi)$ and $(g_2,
\xi)=(gh_2, \xi)$ for some $g\in\langle a, b\rangle$ and $h_1,
h_2\in H_\xi$. If $\pi_{P_i}(h_1h_2)=0$, then $h_1|_{P_i}=h_2|_{P_i}$,
hence $g_1(\zeta)=g_2(\zeta)$ for all $\zeta\in N\cap P_i$ for some neighborhood $N$ of $\xi$. If
$\pi_{P_i}(h_1h_2)=1$, then $g_1(\zeta)\ne g_2(\zeta)$ for all regular
points $\zeta\in N\cap P_i$ for some neighborhood $N$ of $\xi$,
since $h_1h_2|_{P_i}=a|_{P_i}$ and the set of fixed points of $a$
is nowhere dense.

We see that for all regular points $\zeta\in N\cap P_i$, where $N$ is a
sufficiently small neighborhood of $\xi$, the ball $B_\zeta(m)$ of the
orbital graph $\Gamma_\zeta$ is equal to the quotient of the ball
$B_\xi(m)\subset\widetilde\Gamma_\xi$ by the action of the kernel of the
projection $\pi_{P_i}$.
\end{proof}

\begin{corollary}
Every segment of an orbital graph of $G$ is isomorphic to a segment of the orbital graph of a regular point. In particular, for every segment $\Sigma$ of an orbital graph of $G$ an isomorphic copy of $\Sigma$ is contained in every orbital graph of $G$ on some bounded distance $R(\Sigma)$ from every vertex of the orbital graph.
\end{corollary}

\begin{proof}
Let $\xi$ be a non-regular point. Then every segment of $\Gamma_\xi$ is a segment of the limit $\Lambda_i$ of orbital graphs of regular points, hence it is a segment of the orbital graph of a regular point.
\end{proof}

\begin{corollary}
For every oriented segment $\Sigma$ of an orbital graph of $G$ there exist isomorphic copies of $\Sigma$ and $\Sigma^{-1}$ in every oriented orbital graph.
\end{corollary}

\begin{proof}
A copy $\phi(\Sigma)$ of the segment $\Sigma$ is contained in $\Gamma_\xi$. It follows that $\Lambda_i$ contains the copies $\{0\}\times\phi(\Sigma)$ and $\{1\}\times\phi(\Sigma)$ of $\Sigma$. They have opposite orientation, and are contained in a segment $\Sigma'$ of $\Lambda_i$. A copy of the segment $\Sigma'$ is contained in every orbital graph, and inside it we have two copies of $\Sigma$ in opposite orientations.
\end{proof}

\section{Periodicity}
\label{s:periodicity}

\begin{theorem}
\label{th:main}
Let $G$ be a fragmentation of a minimal dihedral group action on a Cantor set $\X$. If there exists a purely non-Hausdorff singularity $\xi\in\X$, then $G$ is periodic.
\end{theorem}

\begin{proof}
We may assume that $\xi$ is a fixed point of $a$.
Let $g\in G$. Let $m$ be the length of $g$ as a product of elements of $A\cup B$. Then for every $\zeta\in\X$ the image $g(\zeta)$
belongs to the ball $B_\zeta(m)$ in the orbital graph $\Gamma_\zeta$,
and is uniquely determined by the labels of the edges of $B_\zeta(m+1)$.

\begin{lemma}
\label{lem:returning}
For every segment $\Sigma$ of an orbital graph of $G$,
a subsegment $\Delta\subset\Sigma$ of edge-length $m$, and a vertex
$v$ of $\Delta$ there exists an embedding $\varphi$ of $\Sigma$
into an orbital graph of a regular point and an integer $k\ge 1$ such
that $g^k(\varphi(v))\in\varphi(\Delta)$.
\end{lemma}

\begin{proof}
Suppose that it is not true for some $\Sigma$, $\Delta$,
$v\in\Delta$, i.e., that for every orbital graph $\Gamma$ of a regular point and an
embedding $\varphi:\Sigma\arr\Gamma$ the sequence $g^k(\varphi(v))$,
$k\ge 1$,
does not come back to $\varphi(\Delta)$. Since for every vertex $u$
the distance from $u$ to $g(u)$ is not more than $m$, the sequence
$g^k(\varphi(v))$, $k\ge 1$, always stays in one of the two connected
components of $\Gamma\setminus\varphi(\Delta)$. It follows that
$g^k(\varphi(v))$ converges to one of the two ends of the graph $\Gamma$. This end is on the same side of $\varphi(\Delta)$ as $g(\varphi(v))$.

There exist embeddings $\Sigma\arr\Gamma_\xi$, in both orientations, where $\xi$ is the purely non-Hausdorff singularity. In particular, there exists an embedding $\varphi:\Sigma\arr\Gamma_\xi$ such that $\varphi(g(v))$ is on the same side of $\varphi(\Delta)$ as $\xi$.

Consider the
corresponding copy $\varphi_0:\Sigma\arr\{\unit\}\times
\Gamma_\xi$ of $\Sigma$ in the ray $\{\unit\}\times
\Gamma_\xi$ of the graph of germs $\widetilde\Gamma_\xi=\Xi$.

Consider the image $\lambda_i\circ\varphi_0(\Sigma)$ of
$\varphi_0(\Sigma)$ in any $\Lambda_i$ under the natural covering $\lambda_i$. It belongs to the ray $\{0\}\times
\Gamma_\xi$ of $\Lambda_i$. Since $\varphi_0(g)$ is closer to $(\unit, \xi)$ than $\varphi_0(\Delta)$, the sequence
$g^k(\lambda_i\circ\varphi_0(v))$ will converge to the infinite end of the ray $\{1\}\times
\Gamma_\xi$ of $\Lambda_i$.

It follows that the sequence $g^k(\varphi_0(v))$ will
converge in $\widetilde\Gamma_\xi$ to an end $\{h\}\times\Gamma_\xi$ different from
$\{\unit\}\times\Gamma_\xi$.

Since $\xi$ is a purely non-Hausdorff singularity, there exists a projection
$\lambda_j:\widetilde\Gamma_\xi\arr\Lambda_j$ such that $\lambda_j(\{h\}\times
\Gamma_\xi)=\lambda_j(\{\unit\}\times\Gamma_\xi)=\{0\}\times\Gamma_\xi$. Then the sequence $\lambda_j(g^k(\varphi_0(w)))$ will
move from one connected component of $\Lambda_j\setminus\lambda_j(\varphi_0(\Delta)$ to another, which is a
contradiction. See Figure~\ref{fig:prooftorsion}, where projections of
$\widetilde\Gamma_\xi$ onto $\Lambda_i$ and $\Lambda_j$ are shown.
\end{proof}

\begin{figure}
\centering
\includegraphics{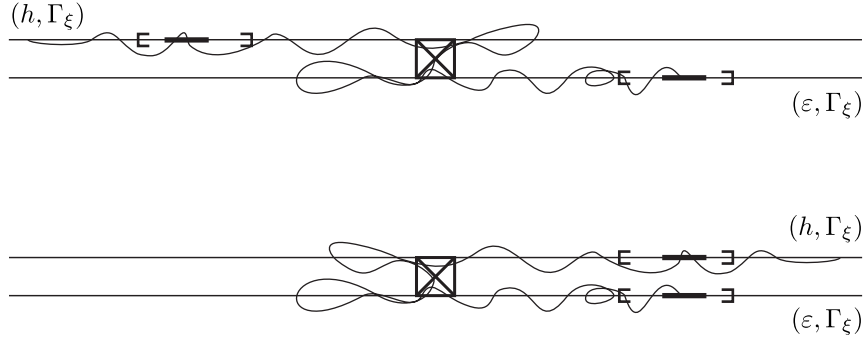}
\caption{Coming back}
\label{fig:prooftorsion}
\end{figure}

Let $\Delta$ be as the Lemma~\ref{lem:returning}, and let $v_0, v_1, \ldots, v_m$ be
the list of its vertices. According to the lemma, there exists a
copy of $\Delta$ in an orbital graph $\Gamma$ of a regular point such
that $g^{k_0}(v_0)\in\Delta$ for some $k_0\ge 1$. Let $\Sigma_0$ be
a sufficiently big segment of $\Gamma$ containing $\Delta$ such that
the $(m+1)$-neigborhood of the sequence $g^{k}(v_0)$ for $k=0, 1, \ldots, k_0$ belongs to $\Sigma_0$. Then $g^{k_0}(v_0)\in\Delta$
in every copy of $\Sigma_0$ in every orbital graph.

Apply now Lemma~\ref{lem:returning} for
$\Sigma=\Sigma_0$ and for the vertex $v_1$ of $\Delta$. We will find
an orbital graph with a copy of $\Sigma_0$ in which both sequences
$g^k(v_0)$ and $g^k(v_1)$ return back to $\Delta$. Therefore there
exists a segment $\Sigma_1$ containing $\Delta$ such that
$g^k(v_0)$ and $g^k(v_1)$ return to $\Delta$ in every orbital graph
containing $\Sigma_1$. Continuing this way we will find a segment
$\Sigma_m$ such that every vertex of $\Delta$ returns inside
$\Sigma_m$ back to $\Delta$ under some positive power of $g$. It
follows that orbit of every vertex of $\Delta\subset\Sigma_m$ is finite
and contained in $\Sigma_m$.

Let $\Gamma$ be an orbital graph of a regular point. By
Proposition~\ref{prop:regularcontainment}, there exists
$R>0$ such that for every vertex $u$ of $\Gamma$ there exists a copy
of $\Sigma_m$ on both sides of $u$ on distances at most $R$. Let
$M$ be the number of vertices of $\Sigma_m$. Then for every vertex $u$ of
$\Gamma$ either the sequence $g^k(u)$ includes a point of one of the
neighboring copies of $\Delta$, or always stays between them. In the
first case the length of the orbit is not more than $M$, in the second
case it is less than $2R+2M$. It follows that the lengths of all
$g$-orbits of vertices of $\Gamma$ are uniformly bounded, hence there
exists $n$ such that $g^n$ acts trivially on the vertices of
$\Gamma$. But the set of vertices of $\Gamma$ is dense in $\X$, so $g^n=\unit$, which finishes the proof of the theorem.
\end{proof}

\begin{proposition}
\label{prop:embeddingfull}
Let $G$ be a fragmentation of a  minimal action of the dihedral group. Then $G$ can be embedded into
the topological full group of a minimal action of $\Z$ on a Cantor set, and hence it is amenable.
\end{proposition}

\begin{proof}
We repeat the argument of~\cite{mattebon:grigochuk}.
Let $\Gamma_\zeta$ be the orbital graph of a regular point
$\zeta\in\X$. It is a bi-infinite chain. Identify the edges of
the chain with integers, so that adjacent edges are identified with integers $n, m$ such that $|n-m|=1$. Let $w_\zeta=(a_n)_{n\in\Z}$
be the corresponding sequence of elements of $\mathcal{P}_A\cup\mathcal{P}_B$ describing the connections between the adjacent vertices, see~\ref{ss:orbitalgraphs}.

Let $\mathcal{W}$ be the set of all sequences $w=\ldots
y_{-1}y_0y_1\ldots$ such that every finite subword of $w$ is a subword
of $w_\zeta$. The set $\mathcal{W}$ is obviously a closed shift-invariant set. Note that for every finite subword $u$ of $w_\zeta$
there exists $R>0$ such that for every $i\in\Z$ there exists $j\in\Z$
such that $|i-j|\le R$ and $x_jx_{j+1}\ldots x_{j+|u|-1}=u$,
see Proposition~\ref{prop:regularcontainment}. This in turn implies that the action of the shift on
$\mathcal{W}$ is minimal. Denote by $\sigma:\mathcal{S}\arr\mathcal{S}$ the shift, which is given by $\sigma(w)(n)=w(n+1)$.

The action of every element $s\in A\cup B$ on a vertex $\eta$ of
$\Gamma_\zeta$ is uniquely determined by the labels of the two edges adjacent to $\eta$. This defines a natural action of $s$ on $\mathcal{W}$ given by the rule
\[
s(w)=\left\{\begin{array}{rl} \sigma(w) & \text{if $\pi_{w(0)}(s)=1$;}\\ \sigma^{-1}(w) & \text{if $\pi_{w(-1)}(s)=1$;}\\ w & \text{otherwise.}\end{array}\right.
\]
If $w$ describes the orbital graph $\Gamma_\zeta$, then $s(w)$ represents the orbital graph of $\Gamma_{s(\zeta)}$.

It is easy to see that the action of $s$ on $\mathcal{S}$ is by an element of the full group $\full(\langle\sigma\rangle, \mathcal{S})$, so that we get an isomorphic embedding of $G$ into $\full(\langle\sigma\rangle, \mathcal{S})$. The result of K.~Juschenko and N.~Monod from~\cite{juschenkomonod} implies now amenability of $G$.
\end{proof}

\begin{proposition}
Suppose that the action of $\langle a, b\rangle$ on $\X$ is
expansive. Let $G$ be a fragmentation of the dihedral group. Then the action of $G$ on $\X$ is also
expansive, and the
group $\alt(G, \X)$ is simple and finitely generated.
\end{proposition}

Note that if $\alt(G, \X)$ is finitely generated, then it is a
subgroup of a fragmentation of the dihedral group with the same groups of germs of points. In particular, if there is a purely non-Hausdorff singularity $\xi\in\X$, then $\alt(G, \X)$ is periodic by Theorem~\ref{th:main} and amenable by Proposition~\ref{prop:embeddingfull}.

\begin{proof}
Let $\delta>0$ be such that $d(g(\zeta), g(\eta))<\delta$ for all
$g\in\langle a, b\rangle$ implies $\zeta=\eta$.

Consider an arbitrary pair $P_1, P_2$ of pieces of the fragmentation $A$. There exist $a_i\in A$ such that $a_i|_{P_i}=a|_{P_i}$. Either $a_1|_{P_2}=a|_{P_2}$, or $a_1|_{P_2}=\unit|_{P_2}$. We also have that either $a_2|_{P_1}=a|_{P_1}$ or $a_2|_{P_1}=\unit|_{P_1}$. It follows that for some $a'\in\{a_1, a_2, a_1a_2\}$ we have $a'|_{P_1\cup P_2}=a|_{P_1\cup P_2}$. We have then $a'|_{\overline{P_1}\cup\overline{P_2}}=a|_{\overline{P_1}\cup\overline{P_2}}$.

It follows that for every two points $\zeta, \eta\in\X$ there exists
$a'\in A$ such that $a'(\zeta)=a(\zeta)$ and
$a'(\eta)=a(\eta)$. Similarly, there exists $b'\in B$ such that $b'(\zeta)=b(\zeta)$ and $b'(\eta)=b(\eta)$.

Consequently, for every $g\in\langle a, b\rangle$
there exists $g'\in G$ such that $g'(\zeta)=g(\zeta)$ and
$g'(\eta)=g(\eta)$.

Suppose that $d(g(\zeta), g(\eta))<\delta$ for all $g\in G$. Then, by
the above, we have $d(g(\zeta), g(\eta))<\delta$ for all $g\in\langle
a, b\rangle$, which implies, by expansivity of $(\langle a, b\rangle,
\X)$, that $\zeta=\eta$. Thus, $(G, \X)$ is also expansive.
Properties of $\alt(G, \X)$ follow now from Theorem~\ref{th:fullgrnek}.
\end{proof}

\section{Intermediate growth}
\label{s:growth}

\subsection{Inverted orbits}
Let $S$ be a finite symmetric generating set of a group $G$ acting on a set $\X$. Choose a point $\xi\in\X$.

Let $g=g_1g_2\ldots g_n$, $g_i\in S$, be a word over $S$ (i.e., an element of the free monoid $S^*$). Following~\cite{BE:permextensions}, we define the \emph{inverted orbit} $\mathcal{O}_\xi(g)$ as the set
\[\mathcal{O}_\xi(g)=\{g_1(\xi), g_1g_2(\xi), g_1g_2g_3(\xi), \ldots, g_1g_2\cdots g_n(\xi)\},\]
where the corresponding products of $g_i$ are considered to be elements of $F$.

\begin{defi}
Let $g=g_1g_2\cdots g_n$ be an element of $S^*$. We say that a pair $(i, j)$ of indices $1\le i<j\le n$ is a \emph{first return} of $\xi$ in the word $g$ if $g_{i+1}g_{i+2}\cdots g_j(\xi)=\xi$ and $g_{k+1}\cdots g_j(\xi)\ne\xi$ for all $i<k<j$. The number $j-i$ is called the \emph{length} of the first return.
\end{defi}

For example, if $g_i(\xi)=\xi$, then $(i-1, i)$ is a first return of length 1. Note that we do not include the cases $g_1g_2\ldots g_j(\xi)=\xi$ as first returns.

\begin{lemma}
\label{lem:firstreturns}
The number of first returns of $\xi$ in $g=g_1g_2\cdots g_n$ is equal to $n-|\mathcal{O}_\xi(g)|$.
\end{lemma}

\begin{proof}
Denote \[\xi_1=g_1(\xi),\quad \xi_2=g_1g_2(\xi),\quad\ldots,\quad \xi_n=g_1g_2\ldots g_n(\xi).\]
A pair $(i, j)$ is a first return if and only if $\xi_i=\xi_j$
and $\xi_k\ne\xi_j$ for $i<k<j$.
It follows that if  $\xi_{i_1}=\xi_{i_2}=\ldots=\xi_{i_m}$ is the
list of all instances of a given element of $\mathcal{O}_\xi(g)$,
and $i_1<i_2<\ldots i_m$, then $(i_1, i_2)$, $(i_2, i_3), \ldots,
(i_{m-1}, i_m)$ are first returns, and that every first return
appears this way exactly once. Note that the number of first returns in this list
is equal to $m-1$. It follows that the total number of first returns
is equal to $n-|\mathcal{O}_\xi(g)|$.
\end{proof}

Denote \[\nu_\xi(n)=\max_{g=g_1g_2\cdots g_n\in S^*}|\mathcal{O}_\xi(g)|.\]
The function $\nu_\xi(n)$ is obviously non-decreasing.

\begin{lemma}
\label{lem:subadditivenu}
For all $m, n\ge 0$ we have
\[\nu_\xi(m+n)\le\nu_\xi(m)+\nu_\xi(n).\]
\end{lemma}

\begin{proof}
Consider a word $g_1g_2\cdots g_{n+m}\in S^*$ of length $n+m$. Then
\begin{multline*}\mathcal{O}_\xi(g_1g_2\cdots g_{n+m})=\{g_1(\xi),
  g_1g_2(\xi), \ldots, g_1g_2\cdots g_m(\xi)\}\cup \\
g_1g_2\cdots g_m\left(\left\{g_{m+1}(\xi),
g_{m+2}g_{m+1}(\xi), \ldots,
g_{m+1}g_{m+2}\cdots g_{m+n}(\xi)\right\}\right)=\\
\mathcal{O}_\xi(g_1g_2\cdots g_m)\cup g_1g_2\cdots
g_m(\mathcal{O}_\xi(g_{m+1}g_{m+2}\cdots g_{m+n})).
\end{multline*}
It follows that
\[|\mathcal{O}_\xi(g_1g_2\cdots g_{n+m})|\le
|\mathcal{O}_\xi(g_1g_2\cdots g_m)|+|\mathcal{O}_\xi(g_{m+1}g_{m+2}\cdots g_{m+n})|\le\nu_\xi(n)+\nu_\xi(m)\]
for every word $g_1g_2\cdots g_{m+n}$, hence $\nu_\xi(m+n)\le\nu_\xi(m)+\nu_\xi(n)$.
\end{proof}

We will also need the following general fact.

\begin{lemma}
\label{lem:doubling}
Suppose that a function $f:\mathbb{N}\arr\mathbb{N}$ is non-decreasing and satisfies $f(n+m)\le f(n)+f(m)$ for all $m, n\in\mathbb{N}$.

Then for all
$n\ge m$, we have $\frac{f(n)}{n}\le \frac{2f(m)}{m}$.
\end{lemma}

\begin{proof}
Then there exist $q\in[0, n/m]\cap\mathbb{N}$ and $r\in 0, 1, \ldots, m-1$ such that $n=qm+r$.
Then
\[f(n)=f(qm+r)\le q f(m)+f(r),\]
hence
\begin{multline*}\frac{f(n)}{n}\le \frac{q f(m)+f(r)}n\le\\
\frac{n f(m)/m+f(m)}{n}=
\frac{f(m)}{m}+\frac{f(m)}{n}=\\
\frac{f(m)}{m}\left(1+\frac{m}{n}\right)\le \frac{2f(m)}{m}\end{multline*}
\end{proof}

\subsection{Inverted orbits of linearly repetitive actions}

Let $G=\langle A\cup B\rangle$ be a fragmentation of a minimal action of the dihedral group $\langle a, b\rangle$.

\begin{proposition}
\label{pr:uppernu}
If the action of $G$ is linearly repetitive and there exists a purely non-Hausdorff singularity $\xi\in\X$, then there exist positive constants $C_1, C_2$ such that
\[\nu_{\zeta}(n)\le C_1n e^{-C_2\sqrt{\log n}}\]
for all $\zeta\in\X$ and $n\ge 1$.
\end{proposition}

\begin{proof}
We may assume that the purely non-Hausdorff singularity $\xi$ is a fixed point of one of the generators $a, b$ (see Corollary~\ref{cor:stabaa}). Let it be $a$. We orient $\Gamma_\xi$ so that the vertex $\xi$ is on the left, and $\Gamma_\xi$ is infinite to the right. If $\Sigma$ is a segment or a ray infinite to the right, then a \emph{left sub-segment} of $\Sigma$ is a sub-segment $\Sigma'$ such that the left end of $\Sigma'$ coincides with the left end of $\Sigma$. In a similar way the notion of a \emph{right sub-segment} is defined.

Let $\{P_0, P_1, \ldots, P_{d-1}\}$ the set of all pieces of the fragmentation $A$ that accumulate on $\xi$. Let $\Lambda_i$ be the limit of the orbital graphs $\Gamma_{\zeta_n}$ for regular points $\zeta_n$ converging to $\xi$ inside $P_i$. Then $\Lambda_i$ is isomorphic to $\Gamma_\xi^{-1}\edge{P_i}\Gamma_\xi$, see Proposition~\ref{pr:Lambdai}. For a natural number $n$, we denote by $P_n$ the label $P_i$ for $i\equiv n\pmod{d}$.

Take an arbitrary left sub-segment $Z_0$ of $\Gamma_\xi$, and define inductively segments $Z_n$ in the following way.

Suppose that we have defined $Z_n$, and let $N$ be the length of $Z_n$. Then there exists a copy $L_n$ of $Z_n^{-1}\edge{P_n}Z_n$ in $\Gamma_\xi$ such that the right end of $L_n$ is at distance at most $KN$ from $\xi$ for some fixed constant $K$ (coming from the estimate of linear repetitivity of orbital graphs). Define $Z_{n+1}$ to be the smallest segment of $\Gamma_\xi$ containing $\xi$ and $L_n$, see Figure~\ref{fig:Zn}.

\begin{figure}
\centering
\includegraphics{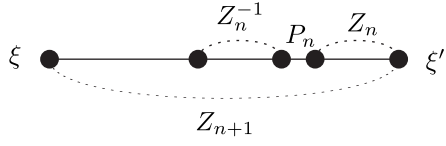}
\caption{Definition of segments $Z_n$}
\label{fig:Zn}
\end{figure}

The length of $Z_{n+1}$ is between $2N$ and $KN$. It follows that length of $Z_n$ is between $2^n|Z_0|$ and $K^n|Z_0|$.

For every $n$ there exists a regular point $\zeta_n$ such that $Z_n$ is isomorphic to a segment $I$ of $\Gamma_{\zeta_n}$ such that the right end of $I$ is equal to $\zeta_n$.
Passing to a convergent subsequence as $n\to\infty$, and using the fact that
$Z_n$ is a right sub-segment of $Z_{n+1}$, we will find a point $\xi'\in\X$ such that for every $n$ there is an isomorphic copy of $Z_n^{-1}\edge{P_n}Z_n$ in $\Gamma_{\xi'}$ with the right end equal to $\xi'$.

Take an arbitrary $n>1$, and let $k_0>0$ be such that $|Z_{k_0-1}|\le n<|Z_{k_0}|$.
Then $2^{k_0-1}|Z_0|<n\le K^{k_0}|Z_0|$, hence there exists positive constants $c_1, c_2$ such that $c_1\log n\le k_0\le c_2\log n$.

Denote by $S$ the generating set $A\cup B$, and let $S^*$ be the free monoid generated by $S$.

Let $M>1$, and consider an arbitrary word $g_1g_2\cdots g_{Mn}\in S^*$.
Let $r_\xi$ and $r_{\xi'}$ be the numbers of first returns of length at most $n$ of $\xi$ and $\xi'$, respectively, in the word $g_1g_2\ldots g_{Mn}$, and let $R_\xi$, $R_{\xi'}$ be the respecitve numbers of first returns of length more than $n$. We have \[|\mathcal{O}_\xi(g_1g_2\cdots g_{Mn})|=Mn-r_\xi-R_\xi,\quad |\mathcal{O}_{\xi'}(g_1g_2\cdots g_{Mn})|=Mn-r_{\xi'}-R_{\xi'}\]
by Lemma~\ref{lem:firstreturns}.

We also have, for every $s=0, 1, \ldots, M-1$,
that the number $|\mathcal{O}_\xi(g_{sn+1}\cdots g_{(s+1)n})|$ is
equal to $n$ minus the number of first returns of $\xi$ in the word $g_{sn+1}g_{sn+2}\cdots g_{(s+1)n}$. Every such a first return is a first return of
$\xi$ in the word $g_1g_2\cdots g_{Mn}$ and its length is not more than $n$. Since the words $g_{sn+1}g_{sn+2}\cdots g_{(s+1)n}$ do not overlap, we get
\begin{multline*}Mn-r_\xi\le \\ |\mathcal{O}_\xi(g_1\cdots
g_n)|+|\mathcal{O}_\xi(g_{n+1}\cdots
g_{2n})|+\cdots+|\mathcal{O}_\xi(g_{(M-1)n+1}\cdots g_{Mn})|\le\\
M\nu_\xi(n),
\end{multline*}
hence
\[|\mathcal{O}_\xi(g_1g_2\cdots g_{Mn})|\le M\nu_\xi(n)-R_\xi,\]
and the same inequality holds for $\xi'$. It follows that
\begin{equation}
\label{eq:ORxi}
|\mathcal{O}_\xi(g_1g_2\cdots g_{Mn})|+|\mathcal{O}_{\xi'}(g_1g_2\cdots g_{Mn})|\le
M(\nu_\xi(n)+\nu_{\xi'}(n))-(R_\xi+R_{\xi'}).
\end{equation}

There exist isomorphic copies of $Z_{k_0+d}^{-1}\edge{P_{k_0}}Z_{k_0+d}$ in $\Gamma_\xi$ (resp., in $\Gamma_{\xi'}$) that are contained in the $K_1n$-neighborhood of $\xi$ (resp., $\xi'$), where $K_1$ is a fixed constant.

Note that if $\Gamma_{\xi'}$ is bi-infinite, we can find two copies of $Z_{k_0+d}^{-1}\edge{P_{k_0}}Z_{k_0+d}$ on both sides of $\xi'$ and both inside the $K_1n$-neighborhood of $\xi'$.
Let us denote the copy of $Z_{k_0+d}^{-1}\edge{P_{k_0}}Z_{k_0+d}$ in $\Gamma_\xi$ by $\Delta$.

The inverted orbit $\mathcal{O}_\xi(g_1g_2\cdots g_{Mn})$ contains at least
\[|\mathcal{O}_\xi(g_1g_2\cdots g_{Mn})|-K_1n\]
elements outside the $K_1n$-neighborhood of
$\xi$.

Suppose that $\zeta=g_1g_2\cdots g_t(\xi)$ is one of them. Consider the path
\[\gamma=(\xi, g_t(\xi), g_{t-1}g_t(\xi), \ldots, g_1g_2\cdots
g_t(\xi))\] in $\Gamma_\xi$. It starts in $\xi$ and traverses
$\Delta$. Let $s_1$ be the smallest index (i.e., the last moment) such that
$g_{s_1}g_{s_1+1}\cdots g_t(\xi)$ is the left end of $\Delta$. Let $s_2$ be the largest index (i.e., the first moment) such that $s_2<s_1$ and $g_{s_2}g_{s_2+1}\cdots g_t(\xi)$ is the right end of $\Delta$. Then
\[\gamma_1=(g_{s_1}\cdots g_t(\xi), g_{s_1-1}g_{s_1}\cdots g_t(\xi), \ldots, g_{s_2}g_{s_2+1}\cdots g_t(\xi))\]
is a path starting in the left end of $\Delta$, ending in the right end of $\Delta$, staying all the time inside $\Delta$, and touching its endpoints only at the first and the last moments.

Recall that the graph of germs $\widetilde\Gamma_\xi$ is isomorphic to
the graph $\Xi$ with the set of vertices
$H_\xi\times\Gamma_\xi$, as it is described in Proposition~\ref{pr:graphgerms}.

Consider the covering
$\lambda_i:\Xi\arr\Lambda_i$, where $i\in\{0, 1, \ldots, d-1\}$ is the residue of $k_0$ modulo $d$, and
let $\widetilde\Delta$ be the lift the central part $Z_{k_0+d}^{-1}\edge{P_{k_0}}Z_{k_0+d}\cong\Delta$  of $\Lambda_i$ to $\Xi$ (recall that $Z_{k_0+d}$ is a left
sub-segment of $\Gamma_\xi$). Let $\widetilde\gamma$
be the lift of $\gamma_1$ to $\widetilde\Delta$ starting in the branch $\{\unit\}\times\Gamma_\xi$ of $\Xi$.

The end of $\widetilde\gamma$ belongs to a branch
$\{h\}\times\Gamma_\xi$ of $\Xi$ for some $h\in H_\xi\setminus\{\unit\}$. There exists $i'\in\{0, 1, \ldots, d-1\}$ such that $\pi_{P_{i'}}(h)=0$, since $\xi$ is a purely non-Hausdorff singularity. Let $k'\in\{k_0+1, k_0+2, \ldots, k_0+d\}$ be such that $k'\equiv i'\pmod{d}$.

Denote by $\Delta'$ the central part of $\Delta$ isomorphic to $Z_{k'}^{-1}\edge{P_{k_0}}Z_{k'}$ (it exists, since $Z_{k'}$ is a left sub-segment of $Z_{k_0+d}$). Denote the full preimage of $\Delta'$ in $\Xi$
by $\widetilde \Delta'$.

The path $\widetilde\gamma$ must enter and exit $\widetilde\Delta'$. It enters
$\widetilde\Delta'$ in the branch $\{\unit\}\times\Gamma_\xi$. Consider the segment
$\widetilde\gamma'$ of $\widetilde\gamma$ from the last entering of $\widetilde\Delta'$ in the branch $\{\unit\}\times\Gamma_\xi$ to the first touching the exit from $\widetilde\Delta'$ after that. The exit
must be in a different branch, since otherwise $\widetilde\gamma$ must
touch the entrance of $\widetilde\Delta'$ one more time (again inside the branch $\{\unit\}\times\Gamma_\xi$). The path $\widetilde\gamma'$ always stays inside $\widetilde\Delta'$ and touches the entrance and the exit of $\widetilde\Delta'$ precisely once each. The image of $\widetilde\gamma'$ in $\Delta$ is of the form
\[(g_{l_1}\cdots g_t(\xi),\quad g_{l_1-1}g_{l_1}\cdots g_t(\xi),\quad\ldots,\quad g_{l_2}\cdots g_t(\xi)),\]
for some $s_2\le l_2<l_1\le s_1$,
where the only point in the path equal to the left end of $\Delta'$ is $g_{l_1}\cdots g_t(\xi_i)$ and the only point equal to the right end of $\Delta'$ is $g_{l_2}\cdots g_t(\xi_i)$.

We have a covering map $\widetilde\Delta'\arr (Z_{k'}^{-1}\edge{P_{k'}}Z_{k'})$ equal to the restriction of $\lambda_{i'}:\Xi\arr\Lambda_{i'}$. Note that the segment of $\Gamma_{\xi'}$ of the form $Z_{k'}^{-1}\edge{P_{k'}}Z_{k'}$ with the right end equal to $\xi'$ is isomorphic as a labeled graph to the corresponding central part of $\Lambda_{i'}$. Let $\psi$ be this isomorphism (from the central part of $\Lambda_{i'}$ to the segment of $\Gamma_{\xi'}$). There are two such isomorphisms, and we choose the one mapping the segment $\{0\}\times Z_{k'}$ of $\Lambda_{i'}$ to the right half (the one containing $\xi'$) of the subsegment $Z_{k'}^{-1}\edge{P_{k'}} Z_{k'}$ of $\Gamma_{\xi'}$.

\begin{figure}
\centering
\includegraphics{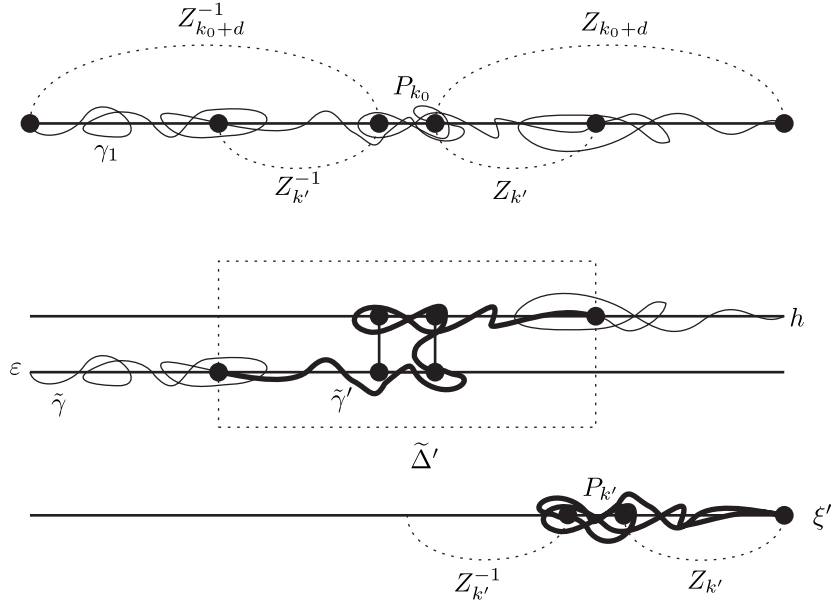}
\caption{Finding a first return of $\xi'$}
\label{fig:growth1}
\end{figure}

If the path $\widetilde\gamma'$ exits $\widetilde\Delta'$ in the branch $\{h\}\times\Gamma_\xi$, then its image under $\psi\circ\lambda_{i'}$ is a path starting in $\xi'$, touching a
vertex of the middle edge $P_{k'}$, and then coming back to
$\xi'$ always staying inside $Z_{k'}^{-1}\edge{P_{k'}}Z_{k'}$. The path $\psi\circ\lambda_{i'}(\widetilde\gamma')$ is equal to
\[(\xi',\quad g_{l_1-1}(\xi'),\quad\ldots,\quad g_{l_2}\cdots g_{l_1-1}(\xi')=\xi'),\]
where $\xi'$ is equal only to the first and to the last vertex in the path, see Figure~\ref{fig:growth1}.

We get a first return $(l_2-1, l_1-1)$ of length at least $2|Z_{k'}|>n$. We know that $g_{l_2}\cdots g_t(\xi)$ is equal to the right end of $\Delta'$, hence the right end of $\Delta'$ and the value of $l_2$ uniquely determines $\zeta$. We see that one such return is produced by at most  $d$ points of the inverted orbit $\mathcal{O}(g_1g_2\cdots g_{Mn})$.

If the path $\widetilde\gamma'$ exits $\widetilde\Delta'$ for the first time in a branch $\{h'\}\times\Gamma_\xi$ labeled by $h'\ne h$, then it has to traverse $\widetilde\Delta'$ at least one more time. It follows that $\widetilde\gamma$ has a subpath $\widetilde\gamma'$ starting at $(h', \xi)\in\Xi$, reaching a preimage of the right end of $\Delta'$, and some time after that coming back to $(h', \xi)$. Take the shortest subpath of this form. Its length is at least $2|Z_{k'}|>n$. Then the image of this subpath under $\lambda_i$ in $\Gamma_\xi$ produces a first return $(l_2-1, l_1-1)$ of $\xi$ of length at least $n$, see Figure~\ref{fig:growth2}. In the same way as in the first case, the image under $g_1g_2\cdots g_{l_2-1}$ of one of the endpoints of the central edge of $\Delta'=(Z_{k'}^{-1}\edge{P_{k_0}} Z_{k'})$ is equal to $\zeta$.
It follows that at most two points of $\mathcal{O}_{\xi_i}(g_1g_2\cdots g_{Mn})$ can produce the same first return this way.

\begin{figure}
\centering
\includegraphics{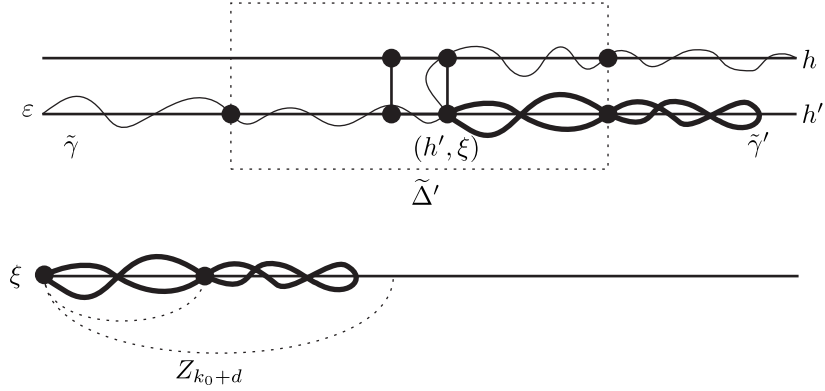}
\caption{Finding a first return of $\xi$}
\label{fig:growth2}
\end{figure}

We see that each $\zeta\in\mathcal{O}_\xi(g_1g_2\cdots g_{Mn})$ outside the $K_1n$-neighborhood of $\xi$ produces either a long first return of $\xi'$ or a long first return of $\xi$, and each such return is produced by at most $d$ points $\zeta$.

The same argument shows that each point of $\mathcal{O}_{\xi'}(g_1g_2\cdots g_{Mn})$ outside of the $K_1n$ neighborhood of $\xi'$ produces a long first return of $\xi'$ or $\xi$, and each such first return is produced by at most $2d$ points of the inverted orbit (we have to multiply by 2, since we have to consider both sides of $\xi'$).

It follows that the number $R_\xi+R_{\xi'}$ of long first returns satisfies
\[R_\xi+R_{\xi'}\ge \frac{1}{3d}\Bigl(|\mathcal{O}_{\xi}(g_1g_2\cdots g_{Mn})|-K_1n+|\mathcal{O}_{\xi'}(g_1g_2\cdots g_{Mn})|-2K_1n \Bigr).\]
Consequently, by~\eqref{eq:ORxi},
\begin{multline*}|\mathcal{O}_{\xi}(g_1g_2\cdots g_{Mn})|+|\mathcal{O}_{\xi'}(g_1g_2\cdots g_{Mn})|\le M(\nu_{\xi}(n)+\nu_{\xi'}(n))-(R_{\xi}+R_{\xi'})\le\\
M(\nu_{\xi}(n)+\nu_{\xi'}(n))-\frac{1}{3d} \Bigl(|\mathcal{O}_{\xi}(g_1g_2\cdots g_{Mn})|+|\mathcal{O}_{\xi'}(g_1g_2\cdots g_{Mn})|-3K_1n\Bigr),
\end{multline*}
hence
\[\frac{3d+1}{3d}\Bigl(|\mathcal{O}_{\xi}(g_1g_2\cdots g_{Mn})|+|\mathcal{O}_{\xi'}(g_1g_2\cdots g_{Mn})|\Bigr)\le M(\nu_\xi(n)+\nu_{\xi'}(n))+\frac{K_1}{d}n.\]
Denote $\nu(n)=\nu_\xi(n)+\nu_{\xi'}(n)$.
Since $g_1g_2\cdots g_{Mn}$ was arbitrary,
we have
\[\frac{3d+1}{3d}\nu(Mn)\le M\nu(n)+\frac{K_1}{d}n.\]

Let us denote $\delta(n)=\frac{\nu(n)}{n}$. Multiplying the last inequality by $\frac{3d}{(3d+1)Mn}$, we get
\[\delta(Mn)\le\frac{3d}{3d+1}\delta(n)+\frac{K_2}{M}\]
for $K_2=3K_1/(3d+1)$.

Let $M=\left\lceil\frac{2(3d+1)K_2}{\delta(n)}\right\rceil$. Then
\[\delta(Mn)\le\frac{3d}{3d+1}\delta(n)+\frac{1}{2(3d+1)}\delta(n)=\frac{6d+1}{6d+2}\delta(n)\]
Denote $\rho=\frac{6d+1}{6d+2}$. It is only important that $0<\rho<1$.

Fix $n_0$, and define inductively a sequence $n_k$ by the rule $n_{k+1}=\left\lceil K_3\delta(n_k)^{-1}\right\rceil n_k$ for $K_3=2(3d+1)K_2$.
Then $\delta(n_{k+1})\le\rho\delta(n_k)$ for every $k$. Taking a bigger $K_3$, if necessary, we may assume that $n_k$ is strictly increasing.

We may assume that $K_3>1$, then \[n_{k+1}=\left\lceil K_3\delta(n_k)^{-1}\right\rceil n_k\le K_4\delta(n_k)^{-1}\cdot n_k,\] for $K_4=K_3+1$, since $\delta(n)\le 1$ for all $n$. Then
\begin{multline*}
n_k\le n_0K_4^k\delta(n_0)^{-1}\delta(n_1)^{-1}\cdots\delta(n_{k-1})^{-1}\le\\
K_4^k\rho^k\delta(n_k)^{-1}\rho^{k-1}\delta(n_k)^{-1}\cdots \rho\delta(n_k)^{-1}=K_4^k\rho^{k(k+1)/2}\delta(n_k)^{-k}
\end{multline*}
rasing the inequality to the power $1/k$, we get
\[n_k^{1/k}\le K_4\rho^{\frac{k+1}{2}}\delta(n_k)^{-1},\]
hence
\[\delta(n_k)\le K_4\rho^{\frac{k+1}{2}}n_k^{-1/k}\le K_5\rho_1^k n_k^{-1/k},\]
for $\rho_1=\sqrt{\rho}$ and $K_5=K_4\rho_1$.

Take an arbitrary $n$. Let $k$ be such that $n_k\le n<n_{k+1}$.
We have
\[\delta(n_k)\le K_5\rho_1^k n_k^{-1/k}\le K_5n_k^{-1/k}\le K_5\left(K_4^{-1}\delta(n_k)n_{k+1}\right)^{-1/k}\le K_5K_4\left(\delta(n_k)n_{k+1}\right)^{-1/k},\]
hence
\[\delta(n_k)^{1+1/k}\le K_6n_{k+1}^{-1/k}<K_6n^{-1/k},\] for $K_6=K_4K_5$,
so
\[\delta(n_k)\le (K_6n^{-1/k})^{\frac{k}{k+1}}\le K_6n^{-1/(k+1)}.\]

Therefore, using Lemma~\ref{lem:doubling}, we get
\[\delta(n)\le 2\delta(n_k)\le 2K_6 n^{-1/(k+1)}\]
and
\[\delta(n)\le 2\delta(n_k)\le 2\delta(n_0)\rho^k.\]

Suppose that $k\le\sqrt{\log n}$. Then
\[\delta(n)\le 2K_6 n^{-1/(k+1)}\le 2K_6n^{-1/(\sqrt{\log n}+1)}=2K_6 e^{-\frac{\log n}{1+\sqrt{\log n}}}\le K_7e^{-K_8\sqrt{\log n}}\]
for some positive constants $K_7$ and $K_8$.

Suppose that $k>\sqrt{\log n}$. Then
\[\delta(n)\le 2\delta(n_0)\rho^k\le 2\delta(n_0)\rho^{\sqrt{\log n}}=2\delta(n_0)e^{\log\rho\sqrt{\log n}}.\]
Taking $C_1=\max\{K_7, 2\delta(n_0)\}$ and $C_2=\min\{K_8, -\log\rho\}$, we get
\[\frac{\nu_\xi(n)}{n}\le\delta(n)\le C_1e^{-C_2\sqrt{\log n}}\]
for all $n$.

Let now $\zeta_0\in\X$ be arbitrary.
Let $n\ge 1$ be a natural number, and let $k$ be such that $|Z_{k-1}|\le n<|Z_k|$.
Then for every $k$ the vertex $\zeta_0$ of $\Gamma_{\zeta_0}$ is contained in a segment $\Sigma$ isomorphic to a segment of the form $Z_k^{-1}IZ_k$, where $|I|\le K_1n$ for some fixed $K_1$. (Recall that $|Z_k|/|Z_{k-1}|$ is bounded.)

We may assume that distance from $\zeta_0$ to the copies of $Z_k$ and $Z_k^{-1}$ in $\Sigma$ is more than $n$. Consider a word $g=g_1g_2\cdots g_{Mn}$.  Split $g$ into subwords $h_1=g_1g_2\cdots g_n$, $h_2=g_{n+1}g_{n+2}\cdots g_{2n}$, \ldots $h_M=g_{(M-1)n+1}g_{(M-1)n+2}\cdots g_{Mn}$.

Suppose that $\zeta_1\in\mathcal{O}_{\zeta_0}(g_1g_2\cdots g_{Mn})\setminus\Sigma$. Then for some $t$ we have $\zeta_1=g_1g_2\cdots g_t(\zeta_0)$. Without loss of generality, let us assume that $\zeta_1$ is to the right of $\zeta_0$.

Represent $t=qn+r$ where $r\in\{0, 1, \ldots, n-1\}$. Then
\[\zeta_1=h_1h_2\cdots h_qg_{qn+1}\cdots g_{qn+r}(\zeta_0).\]
Denote $\zeta_0'=g_{qn+1}\cdots g_{qn+r}(\zeta_0)$. We have the sequence
\[(\zeta_0,\quad\zeta_0',\quad h_q(\zeta_0'),\quad h_{q-1}h_q(\zeta_0'),\quad\ldots, \quad h_1h_2\cdots h_q(\zeta_0')=\zeta_1)\]
such that the distance between consecutive terms in the sequence is not more than $n$. It follows that one of the elements of the sequence belongs to the right subinterval $Z_k$ of $\Sigma$, see Figure~\ref{fig:uniform}. Let $\zeta_2=h_{l}h_{l+1}\cdots h_q(\zeta_0')=g_{(l-1)n+1}g_{(l-1)n+2}\cdots g_t(\zeta_0)$ be the first such point. Then $\zeta_2'=h_{l+1}\cdots h_q(\zeta_0')$ is to the left of $Z_k$. Consider the path
\[(\zeta_2',\quad g_{ln}(\zeta_2'),\quad g_{ln-1}g_{ln}(\zeta_2'), \quad\ldots,\quad g_{(l+1)n+1}\cdots g_{ln}(\zeta_2')=\zeta_2)\]
It passes through the left end $\eta$ of the subinterval $Z_k$ of $\Sigma$. Let $g_sg_{s+1}\cdots g_{(l+1)n}(\zeta_2')$ be the last entry of the sequence equal to $\eta$. Note that $(l-1)n+1\le s\le ln$.

\begin{figure}
\centering
\includegraphics{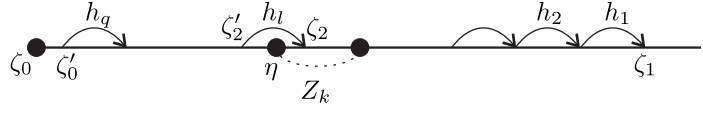}
\caption{A uniform estimate on $\nu(n)$}
\label{fig:uniform}
\end{figure}

Then the path
\[(\xi,\quad g_{s-1}(\xi),\quad \ldots,\quad g_{(l-1)n+1}g_{(l-1)n+2}\cdots g_{s-1}(\xi)),\]
stays inside $Z_k$. It follows that if we map the copy of $Z_k\subset\Sigma$ to the original place of $Z_k$ (the left end of $\Gamma_\xi$), then $\zeta_2$ will be moved to a point belonging to $\mathcal{O}_\xi(g_{(l-1)n+1}g_{(l-1)n+2}\cdots g_{ln})=\mathcal{O}_\xi(h_l)$. It follows that for every value of $l$ there are not more than $\nu_\xi(n)$ possible values of $\zeta_2$.

We have $\zeta_1=h_1h_2\cdots h_{l-1}(\zeta_2)$. It follows that for each $l$ there are not more than $\nu_\xi(n)$ possible values of $\zeta_1$. It follows that the total number of possible values of $\zeta_1$ is not more than $M\nu_\xi(n)$.

We proved that
\[\nu_{\zeta_0}(Mn)\le\left(K_1n+2|Z_k|\right)+2M\nu_\xi(n)\le K_2n+2M\nu_\xi(n),\]
for arbitrary $n$ and $M$, where $K_2>0$ is fixed. Take $n=M$. Then
\[\nu_{\zeta_0}(n^2)\le K_2n+2C_1n^2e^{-C_2\sqrt{\log n}}\le C_1'n^2e^{-C_2'\sqrt{\log n^2}}\]
for some $C_1', C_2'>0$.

For every $n\ge 1$ there exists $k$ such that $k^2/4\le n\le k^2$. Then
\[\nu_{\zeta_0}(n)\le\nu_{\zeta_0}(k^2)\le C_1'k^2e^{-C_2'\sqrt{\log(k^2)}}\le C_1''ne^{-C_2''\sqrt{\log n}}\]
for some positive constants $C_1''$ and $C_2''$, which finishes the proof.
\end{proof}

\subsection{Subexponential growth}

\begin{theorem}
\label{th:growth}
Let $G$ be a fragmentation of a minimal action of a dihedral group on a Cantor set. Suppose that there exists a purely non-Hausdorff singularity $\xi\in\X$, and that the orbital graphs of $G$ are linearly repetitive. Then the growth of $G$ is bounded above by $\exp\left(C_1n e^{-C_2\sqrt{\log n}}\right)$ for some positive constants $C_1$ and $C_2$.
\end{theorem}

\begin{proof}
Choose some point $\xi_0\in\X$. Denote by $I_k$ the segment of $\Gamma_{\xi_0}$ of length $k$ such that $\xi_0$ is its left end. Let $K$ be an arbitrary number larger than the coefficient in the definition of linear repetitiveness.

Choose $\alpha$ such that $1<\alpha<\frac{K+1}{K}$. Define, for $k\ge 1$, $\Sigma_k=I_{\left\lfloor\alpha^k\right\rfloor}$, and denote by $\zeta_k$ the right end of $\Sigma_k$. Then
\[\frac{\alpha^{k+1}-1}{\alpha^k}=\alpha-\alpha^{-k}<
\frac{|\Sigma_{k+1}|}{|\Sigma_k|}<
\frac{\alpha^{k+1}}{\alpha^k-1}=\frac{\alpha}{1-\alpha^{-k}}.\]

Let $g=g_1g_2\cdots g_n\in S^*$ be a word of length $n$ in generators $S=A\cup B$. For an arbitrary regular point $\zeta\in\X$ consider the set $W_\zeta=\{\zeta, g_n(\zeta), g_{n-1}g_n(\zeta), \ldots, g_1g_2\cdots g_n(\zeta)\}$. It is a segment of $\Gamma_\zeta$, since it is the range of a path in $\Gamma_\zeta$. Denote by $l_\zeta$ its length. Let $W_\zeta'$ be the subsegment of $\Gamma_\zeta$ consisting of $W_\zeta$ and the two adjacent edges.
There exists an isomorphic copy $\phi_\zeta(W_\zeta')$ of $W_\zeta'$ in the right half of $\Gamma_{\xi_0}$ such that its left end is on distance at most $Kl_\zeta$ from $\xi_0$ (if $K$ is big enough).

Let $k$ be the smallest positive integer such that $\lfloor\alpha^k\rfloor$ is larger than the distance from $\xi_0$ to the left end of $\phi_\zeta(W_\zeta)$.  Then $\zeta_k$ is to the right of the left end of $\phi_\zeta(W_\zeta)$.
Suppose that $\zeta_k$ is to the right of the right end of $\phi_\zeta(W_\zeta)$. Let $m$ be the distance from $\xi_0$ to the left end of $\phi_\zeta(W_\zeta)$. Then $\lfloor\alpha^k\rfloor\ge m+l$, $\lfloor\alpha^{k-1}\rfloor\le m$, and $m\le Kl$. It follows that
\[\frac{\alpha}{1-\alpha^{-k+1}}>\frac{\lfloor\alpha^k\rfloor}{\lfloor\alpha^{k-1}\rfloor}
\ge\frac{m+l}{m}=1+\frac lm\ge 1+\frac 1K=\frac{K+1}K,\]
which is a contradiction for all $k$ bigger than some fixed $k_0$.

It follows that if $k$ is big enough, then the right end $\zeta_k$ of $\Sigma_k$ belongs to $\phi_\zeta(W_\zeta)$. It follows that there exists $s$ such that $g_s\cdots g_n(\phi_\zeta(\zeta))=\zeta_k$, or
\[\phi_\zeta(\zeta)=g_ng_{n-1}\cdots g_s(\zeta_k).\] It follows that $\phi_\zeta(\zeta)\in\mathcal{O}_{\zeta_k}(g_ng_{n-1}\cdots g_1)$.

Consider the set of all triples $(\phi_\zeta(\zeta), \phi_\zeta(W_\zeta'), \phi_\zeta(g(\zeta)))$ for all regular $\zeta\in\X$. If we know all such triples, then we know $g$, since for every $\zeta\in\X$ an isomorphic copy $\phi_\zeta(W_\zeta')$ of $W_\zeta'$ will appear in the list, and then the pair $\phi_\zeta(\zeta), \phi_\zeta(g(\zeta))$ will determine $g(\zeta)$.

Note that since the length of $W_\zeta$ is not more than $n$, the left end of $\phi(W_\zeta)$ is on distance at most $Kn$ from $\xi_0$, hence the values of $k$ are bounded by $\log(Kn)/\log\alpha\asymp\log n$.

It follows that the number of triples in the list (the length of the list) is not more than $C_1n\log ne^{-C_2\sqrt{\log n}}$ for some positive constants $C_1$ and $C_2$, by Proposition~\ref{pr:uppernu}.

Let us estimate now the number of possible such lists for all words $g\in S^*$ of length $n$. Each entry in the list consists of a point of $I_{(K+1)n}$, a segment of length at most $n$ containing this point, and a point in this segment. It follows that the number of possibilities for each entry of the list is bounded above by $Cn^4$ for some constant $C$. Consequently, the number of possible lists is less than
\begin{multline*}(Cn^4)^{C_1n\log ne^{-C_2\sqrt{\log n}}}=\\
\exp\left(C_1n\log ne^{-C_2\sqrt{\log n}}(4\log n+\log C)\right)\le\\ \exp\left(C_1'n(\log n)^2e^{-C_2\sqrt{\log n}}\right)=\exp\left(C_1'n(e^{2\log \log n-C_2\sqrt{\log n}})\right)\\
\le\exp\left(C_1'ne^{-C_2'\sqrt{\log n}}\right)
\end{multline*}
for some $C_1', C_2'>0$ and all $n$ big enough, since $\frac{\log\log n}{\sqrt{\log n}}\to 0$ as $n\to\infty$. Since the element $g$ is uniquely determined by the corresponding list, this gives the necessary subexponential estimate of the growth of $G$. The group $G$ can not be of polynomial growth, since it is finitely generated, infinite, and periodic, which excludes the possibility of a polynomial growth, by M.~Gromov's Theorem~\cite{gro:gr}.
\end{proof}

\section{Examples}
\subsection{Substitutional systems}
Let $\alb$ be a finite alphabet, and let $\tau:\alb^*\arr\alb^*$ be an endomorphism of the free monoid $\alb^*$. It is uniquely determined by the restriction $\tau:\alb\arr\alb^*$, which is usually called a \emph{substitution}. The associated subshift $\X_\tau\subset\alb^{\Z}$ is the set of all bi-infinite sequences $w$ such that for every finite subword $v$ of $w$ there exists $n\ge 0$ and $x\in\alb$ such that $v$ is a subword of $\tau^n(x)$. It is non-empty if and only if there exists $x\in\alb$ such that the length of $\tau^n(x)$ goes to infinity as $n\to\infty$.

D.~Damanik and D.~Lenz in~\cite{damaniklenz} proved that a substitutional shift is linearly repetitive if and only if it is minimal and gave a criterion of minimality in terms of the substitution.

Let us illustrate how substitutional dynamical systems can be used to construct periodic simple groups of intermediate growth on the example of the Thue-Morse substitution.

Consider the action of the dihedral group from Example~\ref{ex:ThueMorse}. It acts on the shift $\mathcal{S}$ generated by the substitution \[\tau(\zero)=\zero\one,\quad\tau(\one)=\one\zero.\]

Let us fragment the substitution $\tau^2$ by introducing new symbols $t, B, C, D$ and modifying the substitution:
\begin{gather*}
\tau'(\zero)=\zero t\one D\one t\zero ,\quad
\tau'(\one)=\one t\zero D\zero t\one,\\
\tau'(D)=C,\quad
\tau'(C)=B,\quad
\tau'(B)=D,\quad \tau'(t)=t.
\end{gather*} Let $\mathcal{S}'$ be the set of sequences in the shift generated by $\tau'$ that have the letters $\one, \two$ on the even positions, and letters $B, C, D, t$ on the odd positions. We have a natural map $\kappa:\mathcal{S}'\arr\mathcal{S}$ erasing the letters $B, C, D$. One can show that for every $w\in\mathcal{S}$ the set $\kappa^{-1}(w)$ consists of a single element, except for $w$ equal to a shift of one of the two infinite palindromes
\[\ldots \one\zero\zero\one\zero\one\one\zero\;.\;\zero\one\one\zero\one\zero\zero\one\ldots\]
and
\[\ldots\zero\one\one\zero\one\zero\zero\one\;.\;\one\zero\zero\one\zero\one\one\zero\ldots,\]
when $\kappa^{-1}(w)$ has three elements that differ from each other only by the central letter $B, C$, or $D$.

Let $a$ be the transformation of $\mathcal{S}'$ flipping a sequence around the letter on the zeroth position, and let $b$, $c$, $d$, and $t$, respectively, be the transformations flipping a sequence around the letter on the first position if it is $C$ or $D$, $B$ or $D$, $B$ or $C$, and $t$, respectively.
Then the action of $a, b, c, d, t$ on $\mathcal{S}'$ lifts by $\kappa$ to an action on $\mathcal{S}$ in a unique way. The sequences from $\mathcal{S}'$ are naturally interpreted as the orbital graphs of regular points, and the limits $\Lambda_i$ of regular orbital graphs in the case of a singular point of the action of $G=\langle a, b, c, d, t\rangle$ on $\mathcal{S}$. The group $\alt(G, \mathcal{S})$ is a finitely generated simple periodic group of intermediate growth.

\subsection{Groups of polygon rearrangements}

A nice class of examples illustrating Theorem~\ref{th:main} was suggested to the author by Yves de Cornulier. Consider the torus $\R^2/\Z^2$ and two central symmetries $a:x\mapsto -x+v$ and $b:x\mapsto -x$ for some $v\in\R^2/\Z^2$. Suppose that $v$ is represented by $(x, y)\in\R^2$, such that $1, x, y$ are linearly independent over $\Q$. Then, by the classical Kronecker's theorem, the action of $\Z$ generated by the composition $x\mapsto x+v$ of the two symmetries acts minimally on the torus.

Let us split the torus into three $b$-invariant parts $P_1, P_2, P_3$ (e.g., each equal to a union of some polygons) such that the fixed point $0$ of $b$ belongs to the boundary of each of the parts. Consider then the transformations $b_1, b_2, b_3$ (defined up to a set of measure zero) of the torus acting trivially on $P_1$, $P_2$, $P_3$, respectively, and acting as $b$ on their complements. We may also cut the torus open, and represent it as a polygon, so that then $a$, $b_1$, $b_2$, $b_3$ act on the polygon by piecewise isometries. We can lift this action to an action by homeomorphisms of the Cantor set satisfying the conditions of Theorem~\ref{th:main}: one has to double all points lying on the sides of the polygons (except for $0$, which has to remain common to all three parts $P_i$), and then propagate this doubling by the action of the group.

For example, we can consider the group generated by piecewise isometries of the regular hexagon, shown on Figure~\ref{fig:hexagon}. The first transformation $a$ rotates each of the four shown polygons by 180 degrees. The remaining three transformation rotate the shaded areas by 180 degrees around the center of the hexagon and fix the white areas. Theorem~\ref{th:main} implies that this group is periodic, provided the first generator is sufficiently generic (i.e., such that its composition with the 180 degree rotation around the center of the hexagon acts minimally on the torus).

\begin{figure}
\centering
\includegraphics[width=4in]{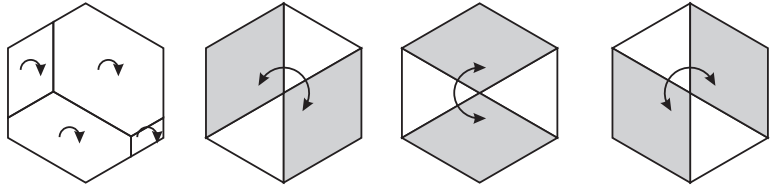}
\caption{A periodic group acting on a hexagon}
\label{fig:hexagon}
\end{figure}

\section{Fragmenting the golden mean dihedral group}
\label{s:fibonacci}

\subsection{The construction}
Let us describe one explicit example of a finitely generated simple periodic group of intermediate growth.

Denote by $\varphi$ the golden mean $\frac{1+\sqrt{5}}2$.
Let $T_\one$ and $T_\two$ be the transformations
\[T_\one(x)=\varphi^{-1} x,\qquad T_\two(x)=1-\varphi^{-2}x\]
of $[0, 1]$. The ranges of $T_\one$ and $T_\two$ are the intervals $[0, \varphi-1]$ and $[\varphi-1, 1]$, respectively. They do not overlap and cover the circle $\R/\Z$.

For every infinite sequence $w=x_1x_2\ldots\in\xo$ over the alphabet $\alb=\{\one, \two\}$, the intersection of the ranges of $T_{x_1}\circ T_{x_2}\circ\cdots\circ T_{x_n}$ is a single point $\xi_w\in \R/\Z$.

Denote by $a$ and $b$ the transformations of the circle $\R/\Z$ given by
\[a(x)=\varphi-x,\qquad b(x)=1-x.\]
Then $ba$ is the rotation $x\mapsto x+\varphi$ of the circle. We get a minimal action of the dihedral group $\langle a, b\rangle$ on the circle.

Direct computations show that
\begin{eqnarray*}
a\circ T_\one(x) &=& T_\one\circ b(x),\\
a\circ T_\two(x) &=& T_\two\circ b(x),
\end{eqnarray*}
and
\begin{eqnarray*}
b\circ T_\one\circ T_\one(x) &=& T_\two(x),\\
b\circ T_\two(x) &=& T_\one\circ T_\one(x),\\
b\circ T_\one\circ T_\two(x) &=& T_\one\circ T_\two\circ b(x),
\end{eqnarray*}
for all $x\in [0, 1]$, where $a$ acts on $[0, \varphi-1]$ by $x\mapsto\varphi-x-1$ and on $[\varphi-1, 1]$ by $x\mapsto\varphi-x$.

We get the following associated action on the sequences $x_1x_2\ldots\in\xo$:
\begin{gather*}
a(\one w)=\one b(w),\qquad a(\two w)=\two b(w),\\
b(\one\one w)=\two w,\quad b(\two w)=\one\one w,\quad b(\one\two w)=\one\two b(w).
\end{gather*}

More precisely, we have a natural map $\kappa:\xo\arr\R/\Z$ mapping a sequence $x_1x_2\ldots$ to the unique intersection point of the ranges of $T_{x_1}\circ T_{x_2}\circ\cdots\circ T_{x_n}$. Then the map $\kappa$ is a semiconjugacy of the transformations $a, b$ acting on $\xo$ with the transformations $a, b$ acting on $\R/\Z$.

As usual, we will identify $\xo$ with the boundary of the rooted tree $\xs$. However, it is more natural to change the metric on the tree in the following way.
The \emph{weight} of the letter $\one$ is
equal to $1$, the weight of the letter $\two$ is equal to $2$. The weight
of a word $v\in\xs$ is equal to the sum of the weights of its letters.

Denote by $L_n$ the set of words of weight $n$. We denote by $L_nv$
for $v\in\xs$ the set of words of the form $uv$ for $u\in
L_n$. Similarly, if $A$ is a subset of $\xs$, then we denote by $A\xo$
the set of all sequences $w\in\xo$ such that a beginning of $w$
belongs to $A$.

We obviously have
\begin{equation}
\label{eq:Ln}
L_n=L_{n-1}\one\sqcup L_{n-2}\two,
\end{equation}
and $L_0=\{\varnothing\}$, $L_1=\{\one\}$, so that $|L_n|$ for $n=0, 1, 2, \ldots$, is the Fibonacci sequence $1, 1, 2, 3, 5, \ldots$.

The transformations $b$ has one fixed point $(\one\two)^\omega$ (encoding the point $1/2$; the point $0$ has two encodings, which are interchanged by $b$). The transformation $a$ has two fixed points $\one(\one\two)^\omega$ and $\two(\one\two)^\omega$ (encoding the points $\varphi/2$ and $(\varphi+1)/2$ of the circle). Let $W_n$, for $n\ge 0$, be the set of sequences starting by $(\one\two)^n\two$ or by $(\one\two)^n\one\one$. Define $P_i=\bigcup_{k=0}^\infty W_{3k+i}$, for $i=0, 1, 2$. The sets $P_i$ form an open partition of $\xo\setminus\{(\one\two)^\omega\}$.

Define, similarly to the Grigorchuk group, homeomorphisms $b_0, c_0, d_0$ of $\xo$ acting trivially on $P_2, P_1, P_0$, respectively, and as $b$ on their complements. Let $a_0$ be the homeomorphism interchanging $\one\one\xo$ with $\two\xo$. More explicitly, the homeomorphisms $a_0, b_0, c_0, d_0$ are given by
\begin{alignat*}{3}
a_0(\one\one w)&=\two w, &\quad a_0(\two w)&=\one\one w, &\quad a_0(\one\two w)&=\one\two w,\\
b_0(\one\one w)&=\two w, &\quad b_0(\two w)&=\one\one w,  &\quad b_0(\one\two w)&=\one\two c_0(w),\\
c_0(\one\one w)&=\two w, &\quad c_0(\two w)&=\one\one w, &\quad c_0(\one \two w)&=\one\two d_0(w),\\
d_0(\one\one w)&=\one\one w, &\quad d_0(\two w)&=\two w, &\quad
d_0(\one \two w)&=\one \two b_0(w).
\end{alignat*}
Note that $a_0$ belongs to the full topological group of $\langle b_0\rangle$.

Let us also fragment the homeomorphism $a$ around its fixed points $\one(\one\two)^\omega$ and $\two(\one\two)^\omega$, in the same way as we fragmented the transformation $b$ around $(\one\two)^\omega$. Namely, define, for every letter $x\in\{a, b, c, d\}$:
\begin{alignat*}{2}
x_1(\one w)&=\one x_0(w), &\quad x_1(\two w)&=\two w, \\
x_2(\one w)&=\one w, &\quad x_2(\two w)&=\two g_0(w).
\end{alignat*}
See Figure~\ref{fig:bcd} for a description of the action of the generators on the boundary of the tree $\xs$.

\begin{figure}
\centering
\includegraphics{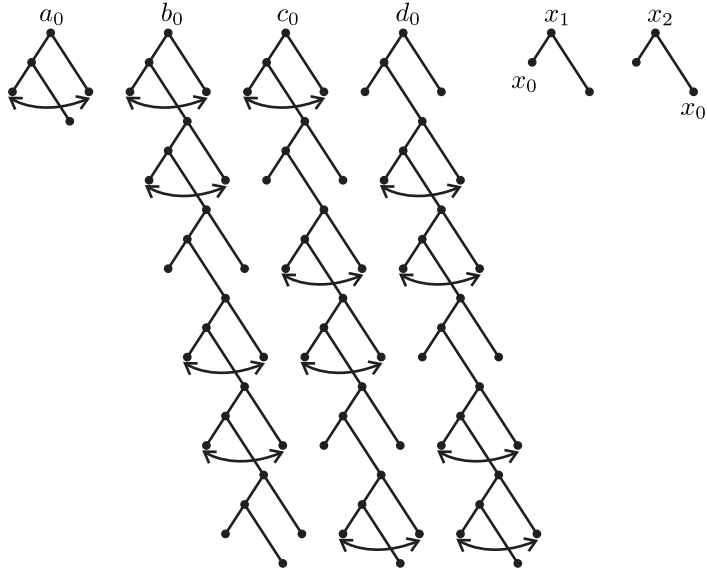}
\caption{The generators of $F$}
\label{fig:bcd}
\end{figure}

The homeomorphisms $b_i, c_i, d_i$ are examples of homeomorphisms defined by \emph{finite asynchronous automata}, see~\cite{grinek_en,grineksu_en}.

Let $F$ be the group generated by $a_i, b_i, c_i, d_i$, $i=0, 1, 2$.
The goal of this section is to prove the following.

\begin{theorem}
\label{th:F}
The group $F$ coincides with its topological full group $\full(F, \xo)$. It is periodic and of intermediate growth. Its derived subgroup $[F, F]$ is simple and has finite index in $F$.
\end{theorem}

Periodicity follows from Theorem~\ref{th:main}.
Note that the set $\{b_0, c_0, d_0, \unit\}$ is a group isomorphic to $(\Z/2\Z)^2$. The element $a_0$ commutes with this subgroup, hence the elements $a_0, b_0, c_0, d_0$ generate a group isomorphic to $(\Z/2\Z)^3$.
The elements $a_1, b_1, c_1, d_1, a_2, b_2, c_2, d_2$ pairwise commute and generate a subgroup isomorphic to $(\Z/2\Z)^6$.
In particular, $F$ is a quotient of the free product $(\Z/2\Z)^3*(\Z/2\Z)^6$. It follows that $F/[F, F]$ is a quotient of $(\Z/2\Z)^9$, hence is finite. In fact, one can show that the abelianization epimorphism $F\arr F/[F, F]$ is induced by the epimorphism $(\Z/2\Z)^3*(\Z/2\Z)^6\arr(\Z/2\Z)^9$, so that $F/[F, F]\cong (\Z/2\Z)^9$, but we do not need it here.

Consequently, it is enough to prove subexponential growth, the equality $F=\full(F, \xo)$, and that $[F, F]=\alt(F, \xo)$.

\subsection{Locally finite groups $\symm_\omega$ and $\alt_\omega$}

\begin{lemma}
\label{lem:splitting}
For every $n\ge 1$ we have
\[\xo=L_n\xo\sqcup L_{n-1}\two\xo.\]
\end{lemma}

\begin{proof} It is true for $n=1$: we have $L_1=\{\one\}$ and $L_0=\{\varnothing\}$ and $\xo=\one\xo\sqcup \two\xo$. Suppose it is true for $n$. Then we have \begin{multline*}\xo=L_n\xo\sqcup L_{n-1}\two\xo=\\
L_n\one\xo\sqcup L_n\two\xo\sqcup L_{n-1}\two=\\ (L_n\one\cup L_{n-1}\two)\xo\cup L_n\two\xo=L_{n+1}\xo\cup L_n\two\xo.
\end{multline*}
\end{proof}

We say that two finite words $v_1$, $v_2$ are \emph{incomparable} if
neither of them is a beginning of the other.
For a set $A$ of pairwise incomparable words, we denote by $\symm(A)$
(resp.\ $\alt(A)$) the group of all (resp.\ even) permutations of $A$
seen as homeomorphisms of $\xo$. If $\alpha$ is a permutation of $A$,
then the corresponding homeomorphism of $\xo$ acts by the rule
$\alpha(vw)=\alpha(v)w$ for $v\in A$, and $\alpha(w)=w$ for $w\notin A\xo$.

The groups $\symm(L_nt)$, $t\in\{\one, \two\}$, are naturally isomorphic to $\symm(L_n)$, where the isomorphis is induced by the bijection $v\mapsto vt$.

By Lemma~\ref{lem:splitting}, the groups $\symm(L_n)$ and
$\symm(L_{n-1}\two)$ act on disjoint subsets of $\xs$, hence they commute. Denote
$\symm(L_n)\oplus\symm(L_{n-1}\two)=\langle\symm(L_n)\cup\symm(L_{n-1}\two)\rangle$. The
group $\alt(L_n)\oplus\alt(L_{n-1}\two)$ is defined in the same way.

Note that
\[\symm(L_n)\oplus\symm(L_{n-1}\two)<\symm(L_{n+1})\oplus\symm(L_n\two),\]
where $\symm(L_n)$ is embedded diagonally into the direct sum by the homomorphism induced by the natural maps
\[v\mapsto v\one:L_n\arr L_n\one\subset L_{n+1}\]
and
\[v\mapsto v\two:L_n\arr L_n\two,\]
and $\symm(L_{n-1}\two)$ is embedded isomorphically to the factor $\symm(L_{n+1})$ by the natural inclusion $L_{n-1}\two\subset L_{n+1}$. The same is true for the embedding \[\alt(L_n)\oplus\alt(L_{n-1}\two)<\alt(L_{n+1})\oplus\alt(L_n\two).\]

Denote by $\symm_\omega$ and $\alt_\omega$ the unions of the groups
$\symm(L_n)\oplus\symm(L_{n-1}\two)$ and $\alt(L_n)\oplus\alt(L_{n-1}\two)$,
respectively, i.e., the direct limit of the described embeddings.

\begin{proposition}
\label{pr:symabel}
The quotient $\symm_\omega/\alt_\omega$ is isomorphic to $(\Z/2\Z)^2$
and is  equal to the set of images of $a_0, a_1, a_2, 1$.
\end{proposition}

\begin{proof}
The quotient
$(\symm(L_n)\oplus\symm(L_{n-1}\two))/(\alt(L_n)\oplus\alt(L_{n-1}\two)$ is
naturally isomorphic to $\Z/2\Z\oplus\Z/2\Z$. It follows from the
description of the embedding
$\symm(L_n)\oplus\symm(L_{n-1}\two)\hookrightarrow\symm(L_{n+1})\oplus\symm(L_n\two)$
that if $(a, b)$ is the image
of an element $g\in\symm(L_n)\oplus\symm(L_{n-1}\two)$ in the quotient
$(\symm(L_n)\oplus\symm(L_{n-1}\two))/(\alt(L_n)\oplus\alt(L_{n-1}\two)$,
then the image of the same element in the quotient
$(\symm(L_{n+1})\oplus\symm(L_n\two))/(\alt(L_{n+1})\oplus\alt(L_n\two))$ is
$(a+b, a)$.

Note that the map $(a, b)\mapsto (a+b, a)$ is an
automorphism of $(\Z/2\Z)^2$, and that the orbit of any non-zero
element of $(\Z/2\Z)^2$ belongs to the cycle $(1, 0)\mapsto (1,
1)\mapsto (0, 1)\mapsto (1, 0)$.

Let $g\in\symm_\omega$, and $n$ be such that
$g\in\symm(L_n)\oplus\symm(L_{n-1}\two)$. Consider then the sequence
$\xi_g=(t_i)_{i\ge n}$ of the images of $g$ in the quotients
$(\symm(L_i)\oplus\symm(L_{i-1}\two))/(\alt(L_i)\oplus\alt(L_{i-1}\two)$. 
Note that the sequence $\xi_g$ is defined only starting from some coordinate. We identify two sequences if they are equal in all coordinates where both of them are defined.

Then for every $g\in\symm_\omega$, the sequence $\xi_g$ is either
equivalent to the constant zero sequence, or to one of the three
shifts of the sequence \[(1, 0), (1, 1), (0, 1), (1, 0), (1, 1), (0,
1), \ldots.\] The sequence $\xi_g$ is equivalent to the constant
zero sequence if and only if $g\in\alt_\omega$. It follows that
$\symm_\omega/\alt_\omega$ is isomorphic to group of equivalence
classes of the sequences
\begin{gather*}((0, 0), (0, 0), (0, 0), \ldots),\\
((1, 0), (1, 1), (0, 1), \ldots),\\
((1, 1), (0, 1), (1, 0), \ldots),\\
((0, 1), (1, 0), (1, 1), \ldots),
\end{gather*}
 which is isomorphic to $(\Z/2\Z)^2$.

The elements $a_0, a_1, a_2$ are equal to the permutations $(\one\one,
\two)\in\symm(L_2)$, $(\one\one\one, \one \two)\in\symm(L_3)$, $(\two\one\one, \two\two)\in\symm(L_4)$,
hence the corresponding sequences $\xi_{a_0}, \xi_{a_1}, \xi_{a_2}$
are
\[
\begin{array}{rrrrrrl}
*, & (1, 0), & (1, 1), & (0, 1), & (1, 0), & (1, 1), &\ldots\\
*, &  *,     & (1, 0), & (1, 1), & (0, 1), & (1, 0), &\ldots\\
*, &  *,     & *,      & (1, 0), & (1, 1), & (0, 1), & \ldots,
\end{array}\]
where asterisk marks the places where the sequence is not defined. We
see that the images of $a_0, a_1, a_2$ are all the non-trivial
elements of $\symm_\omega/\alt_\omega$.
\end{proof}

\subsection{The action of the elements of $F$}

Let $x$ be one of the letters $a, b, c, d$, and let $v\in\xs$. We denote then by $x_v$ the homeomorphism of $\xo$ defined by the rule
\[x_v(w)=\left\{\begin{array}{ll} vx_0(u) & \text{if $w=vu$,}\\ w & \text{if $w\notin v\xo$.}\end{array}\right.\]

Denote then, for every non-negative integer $k$:
\[x_{3k}=x_{(\one\two)^k},\qquad x_{3k+1}=x_{\one(\one\two)^k},\qquad x_{3k+2}=x_{\two(\one\two)^k}.\]
Note that this definition agrees with the original definitions of $x_0, x_1, x_2$.

We have $a_n\in\symm(L_{n+2})<\symm(L_{n+2})\oplus\symm(L_{n+1}\two)$, and
\[b_n=a_nc_{n+3},\qquad c_n=a_nd_{n+3},\qquad d_n=b_{n+3}.\]

It follows that we have the following equalities:
\begin{equation}
\label{eq:eq1}
b_i=a_ic_{i+3}=a_ia_{i+3}d_{i+6}=a_ia_{i+3}b_{i+9}=a_ia_{i+3}a_{i+9}c_{i+12}=\cdots,
\end{equation}
\begin{equation}
\label{eq:eq2}
c_i=a_id_{i+3}=a_ib_{i+6}=a_ia_{i+6}c_{i+9}=a_ia_{i+6}a_{i+9}d_{i+12}=\cdots,
\end{equation}
\begin{equation}
\label{eq:eq3}
d_i=b_{i+3}=a_{i+3}c_{i+6}=a_{i+3}a_{i+6}d_{i+9}=a_{i+3}a_{i+6}b_{i+12}=\cdots,
\end{equation}

We will need the following direct corollary of
equations~\eqref{eq:eq1}--\eqref{eq:eq3}.

\begin{lemma}
\label{lem:deepaction}
Let $n$ be a positive integer, and let $i\in\{0, 1, 2\}$ be such that
$n\equiv i\pmod{3}$. Let $y=b$ if $n-i\equiv 6\pmod{9}$, $y=c$ if
$n-i\equiv 3\pmod{9}$, and $y=d$ if $n-i\equiv 0\pmod{9}$.

Then $y_i=ha_{n-3}d_n$, where $h\in\symm(L_{n-4})\oplus\symm(L_{n-5}\two)$.
\end{lemma}

We say that two sequences $w_1, w_2\in\xo$ are \emph{cofinal} if there exist finite words $v_1, v_2\in\xs$ of equal weight and an infinite word $w\in\xo$ such that $w_1=v_1w$ and $w_2=v_1w$.

Note that no two sequences from the set $R=\{(\one\two)^\omega, \one(\one\two)^\omega, \two(\one\two)^\omega\}$ are cofinal, but every sequence of the form $v(\one\two)^\omega$, where $v\in\xs$, is cofinal to one of the sequences from the set $R$.

It follows directly from the definition of the generators of $F$ that elements of $F$ preserve cofinality classes of sequences. Moreover, the next description of local action of elements of $F$ on $\xo$ is easy to prove by induction on the length of $g$.

\begin{proposition}
Let $g\in F$ be an arbitrary element.

If $u\in\xo$ is not cofinal to any of the elements of $R=\{(\one\two)^\omega, \one(\one\two)^\omega, \two(\one\two)^\omega\}$, then there exists a finite beginning $v_1\in\xs$ of $u$ and a word $v_2\in\xs$ of weight equal to the weight of $v_1$ such that \[g(v_1w)=v_2w\] for all $w\in\xo$.

If $u\in\xo$ is cofinal to an element of $R$ then there exists a finite beginning $v_1\in\xs$ of $u$, a word $v_2\in\xs$ of weight equal to the weight of $v_1$, and an element $h\in\{\unit, b_0, c_0, d_0\}$ such that
\[g(v_1w)=v_2h(w)\]
for all $w\in\xo$.
\end{proposition}

\begin{corollary}
\label{cor:tables}
Let $g\in F$. Then there exist finite sequences \[v_1, v_2, \ldots, v_n, u_1, u_2, \ldots, u_n\in\xs,\] \[h_1, h_2, \ldots, h_n\in\{\unit, b_0, c_0, d_0\},\] such that $\{v_1, v_2, \ldots, v_n\}$ and $\{u_1, u_2, \ldots, u_n\}$ are maximal sets of pairwise incomparable words, weight of $v_i$ is equal to the weight of $u_i$, and for every $w\in\xo$ we have
\[g(v_iw)=u_ih_i(w).\]
\end{corollary}

Note that the element $g$ is uniquely described by the sequences $v_1, v_2, \ldots, v_n$, $u_1, u_2, \ldots, u_n$, and $h_1, h_2, \ldots, h_n$.
\subsection{The recursive structure of the orbital graphs of $F$}
The edges of the orbital graphs of $F$ belong to one of the
following types.
\begin{eqnarray*}
(\one\two)^k\one\one w &\edge{e_{3k}}& (\one\two)^k\two w,\\
\one(\one\two)^k\one\one w &\edge{e_{3k+1}}& \one(\one\two)^k\two w,\\
\two(\one\two)^k\one\one w &\edge{e_{3k+2}}& \two(\one\two)^k\two w,
\end{eqnarray*}
where $w\in\xo$.

The labels $e_k$ stand for the following labelings by the generators
\begin{equation}
\label{eq:edgelabels}
e_{3k+i}=\left\{\begin{array}{rl} a_i, b_i, c_i & \text{if $k=0$,}\\
b_i, c_i & \text{if $k\equiv 0\pmod{3}$ and $k>0$,}\\
b_i, d_i & \text{if $k\equiv 1\pmod{3}$,}\\
c_i, d_i & \text{if $k\equiv 2\pmod{3}$,}\end{array}\right.
\end{equation}
where $i=0, 1, 2$.
Thus, for $k\ge 3$, the label $e_k$ is determined by the residue of $k$ modulo 9.

We will define now graphs $I_k$ with the vertex set
$L_k$. Each of the graphs $I_k$ will be a chain with a fixed choice of the
left/right direction.

For a finite or infinite word $v$, we denote by $I_kv$ the graph
obtained from $I_k$ by appending $v$ to the name of each vertex of
$I_k$.

The graphs $I_0$ and $I_1$ are single vertices
$\varnothing$ and $\one$, respectively.
Define inductively $I_k$ by the rule:
\[I_k=I_{k-2}^{-1}\two\edge{e_{k-2}}I_{k-1}^{-1}\one.\]

\begin{proposition}
\label{pr:endsofIn}
For every infinite word $w$ an orbital graph of $F$ contains $I_nw$. Denote by $P_n$ and $Q_n$ the left and the right endpoints of the chain $I_n$, respectively. Then
\[P_n=(\one\two)^{n/3},\quad \one(\one\two)^{(n-1)/3},\quad\text{or\ }\two(\one\two)^{(n-2)/3},\]
and
\[Q_n=(\two\one)^{n/3},\quad \one(\two\one)^{(n-1)/3},\quad\text{or\ }\one\one(\two\one)^{(n-2)/3},\]
depending on the residue $i=0, 1, 2$ of $n$ modulo $3$.
\end{proposition}

\begin{proof}
Induction on $n$.
\end{proof}

Since minimality of substitutional shifts is equivalent to linear repetitivity (see~\cite{damaniklenz}), the description of the graphs $I_k$ and Theorem~\ref{th:growth} imply the following.

\begin{corollary}
The orbital graphs of $F$ are linearly repetitive, hence the group $F$ has subexponential growth.
\end{corollary}

\subsection{The proof of Theorem~\ref{th:F}}

It follows from Corollary~\ref{cor:tables} that the topological full group $\full(F, \xo)$ is the group of all transformations $g$ that are defined by the rules of the form $g(v_iw)=u_ih_i(w)$, where $v_1, \ldots, v_n, u_1, \ldots, u_n\in\xs$, and $h_1, h_2, \ldots, h_n\in\{b_0, c_0, d_0, \unit\}$, as in Corollary~\ref{cor:tables}.

\begin{proposition}
\label{pr:fullgroup}
The group $F$ coincides with its topological full group.
\end{proposition}

\begin{proof}
Let us prove at first that $F$ contains $\alt_\omega$, i.e.,
that $F$ contains $\alt(L_n)\oplus\alt(L_{n-1}\two)$ for every $n\ge 1$.

The groups $\alt(L_n)$ and $\alt(L_n\two)$ are trivial for $n=0, 1, 2$.
Let us prove by induction that $\alt(L_k)\oplus\alt(L_{k-1}\two)<F$. Suppose that it is true for all $k<n$, and let us prove it for $n$.

Recall that we have
\[\xo=L_n\xo\sqcup L_{n-1}\two\xo\]
and
\[L_n\xo=L_{n-1}\one\xo\sqcup L_{n-2}\two\xo.\]

By Lemma~\ref{lem:deepaction},
for one of the letters $y\in\{b, c, d\}$ and $i'\in\{0, 1, 2\}$ such that $i'\equiv i+1\pmod{3}$, we will have $y_{i'}=ha_{n-2}d_{n+1}$, where $h\in\symm(L_{n-3})\oplus\symm(L_{n-4}\two)<\symm(L_{n-1})\oplus\symm(L_{n-2}\two\xo)$.

The element $a_{n-2}$ interchanges the sets $P_{n-2}\two\xo$ and $Q_{n-1}\one\xo$ and acts trivially on the complement of their union (where $P_n$ and $Q_n$ are as in Proposition~\ref{pr:endsofIn}), since
\[P_{n-2}\two=\one(\one\two)^{(n-3)/3}\two,\quad Q_{n-1}\one=\one\one(\two\one)^{(n-3)/3}\one=\one(\one\two)^{(n-3)/3}\one\one,\]
if $i=0$,
\[P_{n-2}\two=\two(\one\two)^{(n-4)/3}\two,\quad Q_{n-1}\one=(\two\one)^{(n-1)/3}\one=\two(\one\two)^{(n-4)/3}\one\one,\]
if $i=1$, and
\[P_{n-2}\two=(\one\two)^{(n-2)/3}\two,\quad
Q_{n-1}\one=\one(\two\one)^{(n-2)/3}\one=(\one\two)^{(n-2)/3}\one\one,\]
if $i=2$.

The element $d_{n+1}$ preserves each set of the form $v\xo$ for $v\in L_{n-1}\one\cup L_{n-2}\two\cup L_{n-1}\two$, acts identically on each of them, except for the set $Q_{n-2}\two\xo$, since $Q_{n-2}\two$ is one of the sequences $(\two\one)^{(n-2)/3}\two=\two(\one\two)^{(n-2)/3}$, $\one(\two\one)^{(n-3)/3}\two=(\one\two)^{n/3}$, $\one\one(\two\one)^{(n-4)/3}\two=\one(\one\two)^{(n-1)/3}$.

It follows that \[y_{i'}\alt(L_{n-2}\two)y_{i'}=\alt\Bigl((L_{n-2}\two\setminus
\{P_{n-2}\two\})\cup\{Q_{n-1}\one\}\Bigr),\]
see Figure~\ref{fig:alt}.

\begin{figure}
\centering
\includegraphics{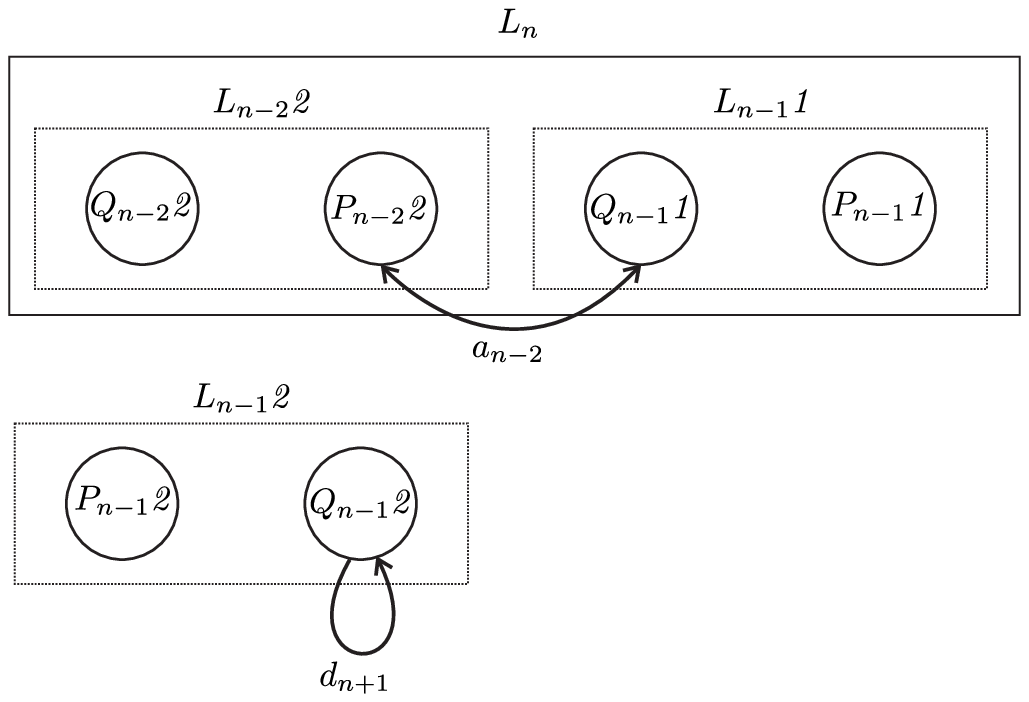}
\caption{Generation of $\alt(L_n)\oplus\alt(L_{n-1}\two)$}
\label{fig:alt}
\end{figure}

It follows that the group generated by $[\alt(L_{n-1}), y_{i'}\alt(L_{n-2}\two)y_{i'}]$ is equal to $\alt(L_n)$.

The elements of $\alt(L_{n-1})<F$ act on $L_{n-1}\one\xo\sqcup L_{n-1}\two\xo$. We have $L_{n-1}\one\subset L_n$, hence $\alt(L_{n-1}\two)$ is contained in the group generated by $\alt(L_{n-1})$ and $\alt(L_n)$, hence it is also contained in $F$.

This finishes the proof of the inclusion $\alt_\omega<F$. 

\begin{lemma}
\label{lem:symmgv}
The group $F$ contains $\symm_\omega$ and all elements of the form $x_v$ for $x\in\{b, c, d\}$ and $v\in\xs$.
\end{lemma}

\begin{proof}
It was shown in Proposition~\ref{pr:symabel} that $\symm_\omega/\alt_\omega$ is equal to the set of images of $a_0, a_1, a_2, \unit$. Since $a_i\in F$, and $\alt_\omega< F$, this implies that $\symm_\omega<F$. In particular, all homeomorphisms $a_n$ belong to $F$.

The relations $b_{n+3}=d_n, c_{n+3}=a_nb_n, d_{n+3}=a_nc_n$ imply by induction that all the homeomorphisms $b_n, c_n, d_n$ belong to $F$.

Let $v\in\xs$, and let $n$ be the weight of $v$. Let $\sigma\in\symm(L_n)$ be the transposition $(v, u)$, where $u\in L_n$ is the unique sequence of the form $(\one\two)^k, \one(\one\two)^k$, or $\two(\one\two)^k$ of weight $n$. Then $x_v=\sigma x_n\sigma$. It follows that $x_v\in F$.
\end{proof}
Lemma~\ref{lem:symmgv} finishes the proof of Proposition~\ref{pr:fullgroup}.
\end{proof}

The next proposition finishes the proof of Theorem~\ref{th:F}.

\begin{proposition}
The derived subgroup $[F, F]$ coincides with $\alt(F, \xo)$.
\end{proposition}

\begin{proof}
By Theorem~\ref{th:fullgrnek} the group $\alt(F, \xo)$ is simple and is contained in every non-trivial normal subgroup of $F$. In particular, $\alt(F, \xo)\le [F, F]$.

Therefore, it is enough to prove that $F/\alt(F, \xo)$ is commutative, i.e., that the generators of $F$ commute modulo $\alt(F, \xo)$.

Note that $\alt(F, \xo)$ obviously contains $\alt_\omega$. Note that $a_n=a_{n+3}$ modulo $\alt_\omega$ for all $n\ge 0$. We have $b_n=a_na_{n+3}b_{n+9}$, $c_n=a_na_{n+6}c_{n+9}$, $d_n=a_{n+3}a_{n+6}d_{n+9}$. It follows that $b_n=b_{n+9}$, $c_n=c_{n+9}$, and $d_n=d_{n+9}$ modulo $\alt(F, \xo)$. We have shown that for every $x\in\{a, b, c, d\}$ we have $x_n=x_{n+9}$ modulo $\alt(F, \xo)$.

If $n>2$, and $v_1, v_2\in L_n$, then there exists an element $\sigma\in\alt(L_n)$ such that $\sigma(v_1)=v_2$. Then, for every letter $x\in\{a, b, c, d\}$, we have $\sigma x_{v_1}\sigma^{-1}=x_{v_2}$. It follows that $x_{v_1}=x_{v_2}$ modulo $\alt(L_n)$.

For arbitrary words $v_1, v_2$ of sufficiently big weight there exist incomparable words $u_1, u_2$ of the same weights as $v_1, v_2$, respectively. Then, for every $x, y\in\{a, b, c, d\}$ we have $x_{v_1}=x_{u_1}$ and $y_{v_2}=y_{u_2}$ modulo $\alt(F, \xo)$, and $[x_{u_1}, y_{u_2}]=1$, hence $[x_{v_1}, y_{v_2}]\in\alt(F, \xo)$.

Let $x_i, y_j$ be generators of $F$, where $x, y\in\{b, c, d\}$, $i, j\in\{0, 1, 2\}$. Then $x_i=x_{i+9k}$, $y_j=y_{j+9k}$ modulo $\alt(F, \xo)$ for all non-negative integers $k$, hence, by the above, $x_i$ and $y_j$ commute in $F/\alt(F, \xo)$.
\end{proof}

\def\cprime{$'$}\def\ocirc#1{\ifmmode\setbox0=\hbox{$#1$}\dimen0=\ht0
  \advance\dimen0 by1pt\rlap{\hbox to\wd0{\hss\raise\dimen0
  \hbox{\hskip.2em$\scriptscriptstyle\circ$}\hss}}#1\else {\accent"17 #1}\fi}
  \def\cprime{$'$}
\providecommand{\bysame}{\leavevmode\hbox to3em{\hrulefill}\thinspace}
\providecommand{\MR}{\relax\ifhmode\unskip\space\fi MR }
\providecommand{\MRhref}[2]{%
  \href{http://www.ams.org/mathscinet-getitem?mr=#1}{#2}
}
\providecommand{\href}[2]{#2}


\begin{thebibliography}{JMMdlS}

\bibitem[AAV13]{amirangelvirag:linear}
Gideon Amir, Omer Angel, and B\'alint Vir\'ag, \emph{Amenability of
  linear-activity automaton groups}, Journal of the European Mathematical
  Society \textbf{15} (2013), no.~3, 705--730.

\bibitem[Adi79]{adian:book}
S.~I. Adian, \emph{The {B}urnside problem and identities in groups}, Ergebnisse
  der Mathematik und ihrer Grenzgebiete [Results in Mathematics and Related
  Areas], vol.~95, Springer-Verlag, Berlin-New York, 1979, Translated from the
  Russian by John Lennox and James Wiegold.

\bibitem[Ady82]{adyan:randomwalk}
S.~I. Adyan, \emph{Random walks on free periodic groups}, Izv. Akad. Nauk SSSR
  Ser. Mat. \textbf{46} (1982), no.~6, 1139--1149, 1343.

\bibitem[Ale72]{al:burn_en}
S.~V. Aleshin, \emph{Finite automata and the {Burnside} problem for periodic
  groups}, Mat.\ Zametki \textbf{11} (1972), 319--328, (in Russian).

\bibitem[AS03]{allouchshallit:sequences}
Jean-Paul Allouche and Jeffrey Shallit, \emph{Automatic sequences. theory,
  applications, generalizations}, Cambridge University Press, Cambridge, 2003.

\bibitem[Bar15]{barth:growthsurvey}
Laurent Bartholdi, \emph{Growth of groups and wreath products}, (preprint
  arXiv:1512.07044), 2015.

\bibitem[BE12]{BE:permextensions}
Laurent Bartholdi and Anna Erschler, \emph{Growth of permutational extensions},
  Invent. Math. \textbf{189} (2012), no.~2, 431--455. \MR{2947548}

\bibitem[BE14a]{BE:givengrowth}
\bysame, \emph{Groups of given intermediate word growth}, Ann. Inst. Fourier
  (Grenoble) \textbf{64} (2014), no.~5, 2003--2036.

\bibitem[BE14b]{bartholdiErschler:imbeddings}
\bysame, \emph{Imbeddings into groups of intermediate growth}, Groups Geom.
  Dyn. \textbf{8} (2014), no.~3, 605--620.

\bibitem[BG13]{baakegrimm}
Michael Baake and Uwe Grimm, \emph{Aperiodic order. {V}ol. 1}, Encyclopedia of
  Mathematics and its Applications, vol. 149, Cambridge University Press,
  Cambridge, 2013, A mathematical invitation, With a foreword by Roger Penrose.

\bibitem[BG{\v{S}}03]{handbook:branch}
Laurent Bartholdi, Rostislav~I. Grigorchuk, and Zoran {\v{S}}uni{\'k},
  \emph{Branch groups}, {Handbook of Algebra, Vol. 3}, North-Holland,
  Amsterdam, 2003, pp.~989--1112.

\bibitem[BKN10]{bkn:amenability}
Laurent Bartholdi, Vadim Kaimanovich, and Volodymyr Nekrashevych, \emph{On
  amenability of automata groups}, Duke Mathematical Journal \textbf{154}
  (2010), no.~3, 575--598.

\bibitem[BM07]{bajorskamacedonska}
B.~Bajorska and O.~Macedo{\'n}ska, \emph{A note on groups of intermediate
  growth}, Comm. Algebra \textbf{35} (2007), no.~12, 4112--4115.

\bibitem[BM08]{bezuglyiMedynets:fullgroup}
S.~Bezuglyi and K.~Medynets, \emph{Full groups, flip conjugacy, and orbit
  equivalence of {C}antor minimal systems}, Colloq. Math. \textbf{110} (2008),
  no.~2, 409--429.

\bibitem[BP06]{buxperez:imgi}
Kai-Uwe Bux and Rodrigo P{\'e}rez, \emph{On the growth of iterated monodromy
  groups}, Topological and asymptotic aspects of group theory, Contemp. Math.,
  vol. 394, Amer. Math. Soc., Providence, RI, 2006, pp.~61--76.

\bibitem[Bri14]{brieussel}
J.~Brieussel, \emph{Growth behaviours in the range $e^{(r^\alpha)}$}, Africa
  Matematika (2014), no.~4, 1143--1163.

\bibitem[B{\v{S}}01]{barthsunik}
Laurent Bartholdi and Zoran {\v{S}}uni{\'k}, \emph{On the word and period
  growth of some groups of tree automorphisms}, Comm. Algebra \textbf{29}
  (2001), no.~11, 4923--4964.

\bibitem[Bur02]{burnside:problem}
W.~Burnside, \emph{On an unsettled question in the theory of discontinuous
  groups}, Quart. J. Pure Appl. Math. \textbf{33} (1902), 230--238.

\bibitem[BV05]{barthvirag}
Laurent Bartholdi and B\'alint Vir\'ag, \emph{Amenability via random walks},
  Duke Math. J. \textbf{130} (2005), no.~1, 39--56.

\bibitem[Cho80]{chou:elementaryamenable}
Ching Chou, \emph{Elementary amenable groups}, Illinois J. Math. \textbf{24}
  (1980), no.~3, 396--407.

\bibitem[CJN16]{ChJN}
M.~Chornyi, K.~Juschenko, and V.~Nekrashevych, \emph{On topological full groups
  of $\mathbb{Z}^d$-actions}, (preprint arXiv:1602.04255), 2016.

\bibitem[dC14]{cornulier:pleinstopologiques}
Yves de~Cornulier, \emph{Groupes pleins-topologiques (d'apr\`es {M}atui,
  {J}uschenko, {M}onod, {$\ldots$})}, Ast\'erisque (2014), no.~361, Exp. No.
  1064, viii, 183--223.

\bibitem[DL06]{damaniklenz}
David Damanik and Daniel Lenz, \emph{Substitution dynamical systems:
  characterization of linear repetitivity and applications}, J. Math. Anal.
  Appl. \textbf{321} (2006), no.~2, 766--780. \MR{2241154}

\bibitem[Ers04]{erschler:nonresfin}
Anna Erschler, \emph{Not residually finite groups of intermediate growth,
  commensurability and non-geometricity}, Journal of Algebra \textbf{272}
  (2004), 154--172.

\bibitem[Ers06]{erschler:piecwise}
Anna Erschler, \emph{Piecewise automatic groups}, Duke Math. J. \textbf{134}
  (2006), no.~3, 591--613.

\bibitem[Ers11]{ershov:gs}
Mikhail Ershov, \emph{Kazhdan quotients of {G}olod-{S}hafarevich groups}, Proc.
  Lond. Math. Soc. (3) \textbf{102} (2011), no.~4, 599--636.

\bibitem[GL02]{grigorchlysenok:burnside}
R.~Grigorchuk and I.~Lysenok, \emph{The {Burnside} problems}, The concise
  handbook of algebra (Alexander~V. Mikhalev and G{\"u}nter~F. Pilz, eds.),
  Kluwer Academic Publishers, Dordrecht, 2002, pp.~111--115.

\bibitem[GLN15]{GLN}
Rostislav Grigorchuk, Daniel Lenz, and Tatiana Nagnibeda, \emph{Schreier graphs
  of {Grigorchuk'}s group and a subshift associated to a non-primitive
  substitution}, (preprint arXiv:1510.00545), 2015.

\bibitem[GN00]{grinek_en}
R.I. Grigorchuk and V.V. Nekrashevich, \emph{{The group of asynchronous
  automata and rational homeomorphisms of the Cantor set.}}, Math. Notes
  \textbf{67} (2000), no.~5, 577--581.

\bibitem[GNS00]{grineksu_en}
Rostislav~I. Grigorchuk, Volodymyr~V. Nekrashevich, and Vitali{\u\i}~I.
  Sushchanskii, \emph{Automata, dynamical systems and groups}, Proceedings of
  the Steklov Institute of Mathematics \textbf{231} (2000), 128--203.

\bibitem[Gol64]{golod:nilalgebras}
E.~S. Golod, \emph{On nil-algebras and finitely approximable {$p$}-groups},
  Izv. Akad. Nauk SSSR Ser. Mat. \textbf{28} (1964), 273--276. \MR{0161878 (28
  \#5082)}

\bibitem[GPS99]{gior:full}
Thierry Giordano, Ian~F. Putnam, and Christian~F. Skau, \emph{Full groups of
  {C}antor minimal systems}, Israel J. Math. \textbf{111} (1999), 285--320.

\bibitem[Gre69]{greenleaf}
F.P. Greenleaf, \emph{Invariant means on topological groups}, Van Nostrand
  Reinhold, New York, 1969.

\bibitem[Gri80]{grigorchuk:80_en}
Rostislav~I. Grigorchuk, \emph{On {Burnside's} problem on periodic groups},
  Functional Anal.\ Appl. \textbf{14} (1980), no.~1, 41--43.

\bibitem[Gri83]{grigorchuk:milnor_en}
Rostislav~I. Grigorchuk, \emph{Milnor's problem on the growth of groups}, Sov.
  Math., Dokl. \textbf{28} (1983), 23--26.

\bibitem[Gri85]{grigorchuk:growth_en}
Rostislav~I. Grigorchuk, \emph{Degrees of growth of finitely generated groups
  and the theory of invariant means}, Math.\ USSR Izv. \textbf{25} (1985),
  no.~2, 259--300.

\bibitem[Gri91]{grigorchuk:kyoto}
\bysame, \emph{On growth in group theory}, Proceedings of the International
  Congress of Mathematicians, Vol.\ I, II (Kyoto, 1990) (Tokyo), Math. Soc.
  Japan, 1991, pp.~325--338.

\bibitem[Gri05]{gri:solvedunsolved}
R.~Grigorchuk, \emph{Solved and unsolved problems aroud one group}, Infinite
  Groups: Geometric, Combinatorial and Dynamical Aspects (T.~Smirnova-Nagnibeda
  L.~Bartholdi, T. Ceccherini-Silberstein and A.~{\.Zuk}, eds.), Progress in
  Mathematics, vol. 248, {Birkh\"auser}, 2005, pp.~117--218.

\bibitem[Gri14]{gri:gap}
Rostislav Grigorchuk, \emph{On the gap conjecture concerning group growth},
  Bull. Math. Sci. \textbf{4} (2014), no.~1, 113--128.

\bibitem[Gro81]{gro:gr}
Mikhael Gromov, \emph{Groups of polynomial growth and expanding maps}, Publ.
  Math.~I.~H.~E.~S. \textbf{53} (1981), 53--73.

\bibitem[Gro87]{gro:hyperb}
\bysame, \emph{Hyperbolic groups}, Essays in Group Theory (S.~M. Gersten, ed.),
  M.S.R.I. Pub., no.~8, Springer, 1987, pp.~75--263.

\bibitem[GS83]{gupta-sidkigroup}
Narain~D. Gupta and Said~N. Sidki, \emph{On the {Burnside} problem for periodic
  groups}, Math.~Z. \textbf{182} (1983), 385--388.

\bibitem[Iva94]{ivanov:burnside}
Sergei~V. Ivanov, \emph{The free {B}urnside groups of sufficiently large
  exponents}, Internat. J. Algebra Comput. \textbf{4} (1994), no.~1-2, ii+308.

\bibitem[JM13]{juschenkomonod}
Kate Juschenko and Nicolas Monod, \emph{Cantor systems, piecewise translations
  and simple amenable groups}, Ann. of Math. (2) \textbf{178} (2013), no.~2,
  775--787.

\bibitem[JMMdlS]{JMMS:extensive}
Kate Juschenko, Nicol\'as {Matte~Bon}, Nicolas Monod, and Mikael de~la Salle,
  \emph{Extensive amenability and an application to interval exchanges},
  (preprint arXiv:1503.04977).

\bibitem[JNdlS13]{YNS}
Kate Juschenko, Volodymyr Nekrashevych, and Mikael de~la Salle,
  \emph{Extensions of amenable groups by recurrent groupoids}, (preprint
  arXiv:1305.2637), 2013.

\bibitem[KLS15]{kellendonLenz:aperiodicorder}
Johannes Kellendonk, Daniel Lenz, and Jean Savinien (eds.), \emph{Mathematics
  of aperiodic order}, Progress in Mathematical Physics, vol. 309,
  Birkh\"auser/Springer, Basel, 2015.

\bibitem[KP13]{kassabovpak:oscillating}
Martin Kassabov and Igor Pak, \emph{Groups of oscillating intermediate growth},
  Ann. of Math. (2) \textbf{177} (2013), no.~3, 1113--1145.

\bibitem[Kri80]{krieger:homgroups}
Wolfgang Krieger, \emph{On a dimension for a class of homeomorphism groups},
  Math. Ann. \textbf{252} (1979/80), no.~2, 87--95. \MR{593623 (82b:46083)}

\bibitem[Lys96]{lysenok:burnside}
I.~G. Lys{\"e}nok, \emph{Infinite {B}urnside groups of even period}, Izv. Ross.
  Akad. Nauk Ser. Mat. \textbf{60} (1996), no.~3, 3--224. \MR{1405529}

\bibitem[Man12]{mann:growth}
Avinoam Mann, \emph{How groups grow}, London Mathematical Society Lecture Note
  Series, vol. 395, Cambridge University Press, Cambridge, 2012. \MR{2894945}

\bibitem[Mat06]{matui:fullI}
Hiroki Matui, \emph{Some remarks on topological full groups of {C}antor minimal
  systems}, Internat. J. Math. \textbf{17} (2006), no.~2, 231--251.

\bibitem[Mat12]{matui:etale}
\bysame, \emph{Homology and topological full groups of \'etale groupoids on
  totally disconnected spaces}, Proc. Lond. Math. Soc. (3) \textbf{104} (2012),
  no.~1, 27--56.

\bibitem[Mat13]{matui:expgrowth}
\bysame, \emph{Some remarks on topological full groups of {C}antor minimal
  systems {II}}, Ergodic Theory Dynam. Systems \textbf{33} (2013), no.~5,
  1542--1549.

\bibitem[Mat15]{matui:fullonesided}
\bysame, \emph{Topological full groups of one-sided shifts of finite type}, J.
  Reine Angew. Math. \textbf{705} (2015), 35--84.

\bibitem[MB14]{mattebon:liouville}
Nicol{\'a}s Matte~Bon, \emph{Subshifts with slow complexity and simple groups
  with the {L}iouville property}, Geom. Funct. Anal. \textbf{24} (2014), no.~5,
  1637--1659.

\bibitem[MB15]{mattebon:grigochuk}
\bysame, \emph{Topological full groups of minimal subshifts with subgroups of
  intermediate growth}, J. Mod. Dyn. \textbf{9} (2015), 67--80.

\bibitem[Mil68]{milnor:5603}
John Milnor, \emph{Problem 5603}, Amer. Math. Monthly \textbf{75} (1968),
  685--686.

\bibitem[MK15]{kourovka}
V.~D. Mazurov and E.~I. Khukhro, \emph{Unsolved problems in group theory. the
  kourovka notebook. no. 18}, (arXiv:1401.0300), 2015.

\bibitem[NA68]{novikovadjan}
P.~S. Novikov and S.~I. Adjan, \emph{Infinite periodic groups. {I}}, Izv. Akad.
  Nauk SSSR Ser. Mat. \textbf{32} (1968), 212--244. \MR{0240178}

\bibitem[Nek05]{nek:book}
Volodymyr Nekrashevych, \emph{Self-similar groups}, Mathematical Surveys and
  Monographs, vol. 117, Amer. Math. Soc., Providence, RI, 2005.

\bibitem[Nek15]{nek:fullgr}
\bysame, \emph{Simple groups of dynamical origin}, (preprint arXiv:1511.08241),
  2015.

\bibitem[Ol{\cprime}80]{olshanskii:nonamenable}
A.~Ju. Ol{\cprime}{\v{s}}anski{\u\i}, \emph{On the question of the existence of
  an invariant mean on a group}, Uspekhi Mat. Nauk \textbf{35} (1980),
  no.~4(214), 199--200.

\bibitem[Ol{\cprime}82]{olshanski:monster}
A.~Yu. Ol{\cprime}shanski{\u\i}, \emph{Groups of bounded period with subgroups
  of prime order}, Algebra i Logika \textbf{21} (1982), no.~5, 553--618.
  \MR{721048}

\bibitem[Ol{\cprime}91]{olshanski:book}
\bysame, \emph{Geometry of defining relations in groups}, Mathematics and its
  Applications (Soviet Series), vol.~70, Kluwer Academic Publishers Group,
  Dordrecht, 1991, Translated from the 1989 Russian original by Yu. A.
  Bakhturin.

\bibitem[Ol{\cprime}93]{olshanski:residualing}
\bysame, \emph{On residualing homomorphisms and {$G$}-subgroups of hyperbolic
  groups}, Internat. J. Algebra Comput. \textbf{3} (1993), no.~4, 365--409.

\bibitem[Pat88]{paterson:amenability}
Alan L.~T. Paterson, \emph{Amenability}, Mathematical Surveys and Monographs,
  vol.~29, American Mathematical Society, Providence, RI, 1988.

\bibitem[Rip82]{rips:cancellation}
E.~Rips, \emph{Generalized small cancellation theory and applications. {I}.
  {T}he word problem}, Israel J. Math. \textbf{41} (1982), no.~1-2, 1--146.

\bibitem[Sha06]{shalom:congress}
Yehuda Shalom, \emph{The algebraization of {K}azhdan's property ({T})},
  International {C}ongress of {M}athematicians. {V}ol. {II}, Eur. Math. Soc.,
  Z\"urich, 2006, pp.~1283--1310.

\bibitem[{\v S}un07]{sunic:hausdorff}
Zoran {\v S}uni{\'c}, \emph{Hausdorff dimension in a family of self-similar
  groups}, Geometriae Dedicata \textbf{124} (2007), 213--236.

\bibitem[Su{\v{s}}79]{sushch:periodic}
V.~{\=I}. Su{\v{s}}{\v{c}}ans{\cprime}ki{\u\i}, \emph{Periodic {$p$}-groups of
  permutations and the unrestricted {B}urnside problem}, Dokl. Akad. Nauk SSSR
  \textbf{247} (1979), no.~3, 557--561.

\bibitem[vN29]{vneumann:masses}
John von Neumann, \emph{Zur allgemeinen {Theorie} des {Masses}}, Fund. Math.
  \textbf{13} (1929), 73--116 and 333, \emph{Collected works}, vol.\ I, pages
  599--643.

\bibitem[Vor12]{vorob:schreiergraphs}
Yaroslav Vorobets, \emph{Notes on the {Schreier} graphs of the {Grigorchuk}
  group}, Dynamical systems and group actions (L.~Bowen et~al., ed.), Contemp.
  Math., vol. 567, Amer. Math. Soc., Providence, RI, 2012, pp.~221--248.

\bibitem[Wag94]{wagon:banachtarski}
Stan Wagon, \emph{The {B}anach-{T}arski paradox}, Cambridge University Press,
  1994.

\bibitem[Wil89]{wilf:algorithms}
Herbert~S. Wilf, \emph{Combinatorial algorithms: an update}, CBMS-NSF Regional
  Conference Series in Applied Mathematics, vol.~55, Society for Industrial and
  Applied Mathematics (SIAM), Philadelphia, PA, 1989.

\bibitem[Zel90]{zelmanov:burnside}
E.~I. Zel{\cprime}manov, \emph{Solution of the restricted {B}urnside problem
  for groups of odd exponent}, Izv. Akad. Nauk SSSR Ser. Mat. \textbf{54}
  (1990), no.~1, 42--59, 221.

\end{thebibliography}
\end{document}